\documentclass[12pt]{amsart}
\usepackage{cite}

\usepackage{amssymb}

\usepackage{enumitem}

\usepackage{graphicx}

\makeatletter
\@namedef{subjclassname@2020}{%
  \textup{2020} Mathematics Subject Classification}
\makeatother

\usepackage[T1]{fontenc}
\newtheorem{theorem}{Theorem}[section]

\newtheorem{lemma}[theorem]{Lemma}

\theoremstyle{definition}

\numberwithin{equation}{section}

\DeclareMathOperator{\mmod}{\:mod}
\DeclareMathOperator{\capJ}{J}
\DeclareMathOperator{\capL}{L}

\def\a{v}
\def\r{\rho}
\def\c{w}

\begin{document}

%%%%% To ease editing, for IMPAN journals add:

\baselineskip=17pt

%%%%%%%%%%%%%%%%

\title[On the Factorization of lacunary polynomials]{On the Factorization of lacunary polynomials}

\author{Michael Filaseta}
\address{Mathematics Department\\ University of South Carolina\\
Columbia, SC  29208\\
USA}
\email{filaseta@math.sc.edu}

\date{}

\begin{abstract}
This paper addresses the factorization of polynomials of the form 
$F(x) = f_{0}(x) + f_{1}(x) x^{n} + \cdots + f_{r-1}(x) x^{(r-1)n} + f_{r}(x) x^{rn}$ 
where $r$ is a fixed positive integer and the $f_{j}(x)$ are fixed polynomials
in $\mathbb Z[x]$ for $0 \le j \le r$.  We provide an efficient method for showing that for $n$ sufficiently large and reasonable conditions on the $f_{j}(x)$, 
the non-reciprocal part of $F(x)$ is either $1$ or irreducible.  We illustrate the approach including giving two examples that arise from 
trace fields of hyperbolic $3$-manifolds.
\end{abstract}



\maketitle

\centerline{\it Dedicated to the fond and inspirational memory of A.~Schinzel.}

\section{Introduction}

The main goal of this paper is to explain how to obtain information about the factorization of polynomials of the form 
\begin{equation}\label{generalfformeqone}
F(x) = \sum_{j=0}^{r} f_{j}(x) x^{nj}, 
\end{equation}
where $r$ is an integer greater than or equal to $1$, the polynomials $f_{0}(x), \ldots,$ $f_{r}(x)$ are in $\mathbb Z[x]$ and $n$ is a sufficiently large positive integer.  
In line with the work of A.~Schinzel described in the next paragraph, for a non-zero Laurent polynomial $F(x) \in \mathbb Q[x,x^{-1}]$, we define
$\capJ F$ as $x^{k} F(x)$ where $k \in \mathbb Z$ is chosen so that $x^{k} F(x) \in \mathbb Q[x]$ and the constant term of
$x^{k} F(x)$ is non-zero.
Our particular interest is in determining general conditions on the polynomials $f_{0}(x), \ldots, f_{r}(x)$ for which the \textit{non-reciprocal part} of $F(x)$ 
is either $1$ or irreducible provided only that $n$ is sufficiently large.  
The non-reciprocal part of a polynomial $F(x) \in \mathbb Z[x]$ is defined as follows.  
The reciprocal of a non-zero polynomial $f(x) \in \mathbb Q[x]$ is $\tilde{f}(x) = \capJ f(1/x) = x^{\deg f} f(1/x)$. 
A non-zero polynomial $f(x) \in \mathbb Q[x]$ is \emph{reciprocal} if $f(x) = \pm \tilde{f}(x)$.   
The non-reciprocal part of $f(x) \in \mathbb Z[x]$ is $f(x)$ removed of each
irreducible reciprocal factor in $\mathbb Z[x]$ with a positive leading coefficient.
The latter is to be interpreted as removing these factors to the multiplicity to which 
they appear so that only irreducible non-reciprocal factors remain.  
ln Schinzel's work (cf.~\cite{lac1,lac3,lac12}), the non-reciprocal part of $f(x)$ is denoted $\capL f(x)$, 
though we will avoid using this notation in the remainder of this paper.

This paper is motivated in part by the tremendous amount of work A.~Schinzel produced on the factorization of lacunary (or sparse) polynomials 
(cf.~\cite{reduciblelac, schlac,  lac1, lac2, lac3, lac4, lac5, lac6, lac7, lac7errata, lac8, lac9, lac10, lac11, lac12, lac12corr, schwoj}).
In particular, his early papers on the subject gave a description of how the \textit{non-reciprocal part} of a lacunary polynomial $f(x)$ factors,
and later he extended his results to the \textit{non-cyclotomic part} of a lacunary polynomial $f(x)$ (that is $f(x)$ removed of all of its cyclotomic polynomial
factors) in the case that $f(x)$ is not reciprocal.  The latter was made possible by advancements by E.~Bombieri and U.~Zannier described
in an appendix by Zannier in \cite{schbook2} (also, see the work of E.~Bombieri , D.~W.~Masser and U.~Zannier \cite{bmz}).

The paper is also motivated by private communications of the author with S.~Garoufalidis.  
S.~Garoufalidis pointed out the significance of determining the factorization of polynomials of the form \eqref{generalfformeqone}  
to determine the trace field associated with $\Gamma \subset \text{PSL}_{2}(\mathbb C)$.  
In particular, he pointed out the trace field of the Whitehead link $W_{m/n}$ is $\mathbb Q(x-1/x)$ where $x$ is some root of
\[
(x(x+1))^m x^{4n}-(x-1)^m
\]
(see W.~D.~Neumann and A.~W.~Reid \cite{neurei});
the trace field of the Whitehead link for the complement of the $(-2,3,3+2n)$ pretzel knot $K_n$ is $\mathbb Q(x+1/x)$ where $x$ is some root of
\[
x^{4n}+\left(-x^3+4x^2-8+4x^{-2}-x^{-3}\right)x^{2n}+1
\]
(see M.~L.~Macasieb and T.~W.~Mattman \cite{macmat});
and the trace field under $-a/b$ Dehn fillings of the complement of the $4_1$ knot is $\mathbb Q(x+1/x)$ where $x$ is some root of
\[
x^{4b} - x^{2b} - x^{a} - 2 - x^{-a} - x^{-2b} + x^{-4b}
\]
(S.~Garoufalidis, private communication).  

The first of these three Laurent polynomials is easier to approach given current information in the literature, and 
S.~Garoufalidis and the author \cite{filgar} show that if $m$ is a fixed odd positive integer,
there is an $N(m)$ such that if $n \ge N$ is an integer relatively prime to $m$, then the polynomial given by
$(x(x+1))^m x^{4n}-(x-1)^m$ is $x^{2}+1$ times an irreducible polynomial (over $\mathbb Q$).  
The other two Laurent polynomials are more difficult to find literature to exploit.  
There are two reasons for this.  First, if $F(x)$ is one of these two Laurent polynomials, then $\capJ F(x)$ is reciprocal.  
Based on current knowledge, determining the complete factorization of reciprocal polynomials in a general form is more difficult 
than the nonreciprocal case.  
For example, with regard to Schinzel's work described above, 
he requires a condition that $\capJ F(x)$ is reciprocal when discussing 
the factorization of the non-cyclotomic part of $F(x)$ (see \cite[Theorem~2]{lac12}).
Second, the first of the three Laurent polynomials is the special case of \eqref{generalfformeqone} with $r = 1$
which has been investigated more fully in the literature (eg., \cite{FFK, flps, SSS, schcover}), whereas the second and third
Laurent polynomials above are cases where $r \ge 2$ in \eqref{generalfformeqone} for which similar literature is lacking.

In general, what is reasonable to expect given the current literature on the subject is that 
except for exceptional cases, perhaps explicitly laid out, one can show
that $F(x)$ in \eqref{generalfformeqone} for large $n$ has non-reciprocal part which is either
identically $1$ or irreducible.  In the case that $F(x)$ itself is reciprocal, this would necessarily mean
that the non-reciprocal part is $1$ and that each irreducible factor is reciprocal.  
The latter is of some significance.  Given an arbitrary non-zero polynomial $g(x) \in \mathbb Q[x]$, 
the polynomial $g(x) \tilde{g}(x)$ is reciprocal.  Thus, apriori, given only that $F(x)$ is reciprocal, we
cannot deduce that its irreducible factors are themselves reciprocal.  Knowing the irreducible factors 
are reciprocal also provides information about the Galois group of a number field generated by the roots
of an irreducible factor of the polynomial. 
More precisely, if $f(x)$ is an irreducible reciprocal polynomial of degree $n > 1$, then $n = 2m$
for some integer $m \ge 1$ and the associated Galois group is isomorphic to a subgroup of all signed permutations 
on $m$ objects and, hence, has order dividing $2^{m} m!$ (see \cite{dds}). 

With the above in mind, our main result is the following.

\begin{theorem}\label{themainthm}
Let $r \ge 1$ be an integer, and
\[
F(x,y) = f_{0}(x) + f_{1}(x) y + \cdots + f_{r}(x) y^{r} \in \mathbb Z[x,y],
\]
with 
\[
\gcd{}_{\mathbb Z[x]}\big(f_{0}(x), f_{1}(x), \ldots, f_{r}(x) \big) = 1,
\quad
f_{0}(0) \ne 0,
\quad \text{and} \quad
f_{r}(x) \ne 0.
\]
Let $n$ be sufficiently large.  
Then the non-reciprocal part of $F(x,x^{n})$ is not reducible if, 
whenever we write $n = k \ell + t$, where $k, \ell$ and $t$ are integers satisfying
\begin{equation}\label{mainthmeq}
k \ge 2 \max_{0 \le j \le r} \{ \deg(f_{j}) \}  \quad \text{ and } \quad
0 \le t < \dfrac{k}{r(r+1)}
\end{equation}
the polynomial
\[
f_{0}(x) 
+ f_{1}(x) x^{t} y^{\ell}
+ \cdots
+ f_{r-1}(x) x^{(r-1)t} y^{(r-1)\ell}
+ f_{r}(x) x^{rt} y^{r \ell}
\]
is a power of $x$ (possibly $x^{0}$) times an irreducible polynomial in $\mathbb Z[x,y]$,
and whenever we write $n = k \ell - t$, where $k, \ell$ and $t$ are integers satisfying
\eqref{mainthmeq}, the polynomial
\[
f_{0}(x) x^{rt}
+ f_{1}(x) x^{(r-1)t} y^{\ell}
+ \cdots
+ f_{r-1}(x) x^{t} y^{(r-1)\ell}
+ f_{r}(x) y^{r \ell}
\]
is a power of $x$ (possibly $x^{0}$) times an irreducible polynomial in $\mathbb Z[x,y]$.  
\end{theorem}

The subscript on the greatest common divisor indicates the it is taken over the ring $\mathbb Z[x]$, 
so the polynomials $f_{0}(x), f_{1}(x), \ldots, f_{r}(x)$ are not all divisible by an element of $\mathbb Z[x]$ other than
one of the units $\pm 1$.  
The terminology ``not reducible'' instead of ``irreducible'' is to allow for the possibility that 
 the non-reciprocal part of $F(x,x^{n})$ is $1$.  
 Note also that the conditions in the theorem imply $\ell \ge 1$.  

For clarification, the work of A.~Schinzel beginning with \cite{lac1} (see also \cite{filpap}) already gives, in theory, 
an effective way of determining whether the non-reciprocal parts of $F(x,x^{n})$ in Theorem~\ref{themainthm}, 
and even more general polynomials, are reducible.  
The goal in this paper is to provide a more practical approach, 
a method in particular of efficiently handling polynomials like the Laurent polynomials described earlier.

In Section~\ref{twovarsection}, we elaborate on a method to show that polynomials in two variables of the form appearing
throughout Theorem~\ref{themainthm} are irreducible (or a power of $x$ times an irreducible polynomial).  
In Section~\ref{prelimsect}, we establish a preliminary result for Theorem~\ref{themainthm} along the lines given in \cite{FFK}.
A simpler variant of Theorem~\ref{themainthm} and Theorem~\ref{themainthm} itself will be established in Section~\ref{multitoonesect}.  
In Section~\ref{sectapplications}, 
we will look at consequences of Theorem~\ref{themainthm}.
We provide one example there involving the non-reciprocal polynomial
\begin{align*}
x^{6n} + (x+1) &x^{5n+1} + 2x^{4n} + (x^4 - x^3 - x^2 - 2x - 2)x^{3n-2} \\
&\quad+ 2x^{2n} + (x+1) x^{n-2} + 1,
\end{align*}
showing that for $n$ sufficiently large, the polynomial is irreducible.  
We then give examples corresponding to the last two
Laurent polynomials connected to trace fields indicated above,
and we obtain results showing that, for large enough exponent variables,
the associated polynomials have irreducible factors that are reciprocal.  

Before leaving the introduction, we comment that under the assumption of Lehmer's conjecture on the Mahler measure
of a non-cyclotomic irreducible polynomial in $\mathbb Z[x]$, each irreducible non-cyclotomic factor of 
polynomials of the form \eqref{generalfformeqone}, with $\gcd(f_{0}(x), f_{1}(x), \ldots, f_{r}(x)) = 1$, has degree $> c n$, 
where $c$ is a positive constant depending only
on the polynomials $f_{0}(x), f_{1}(x), \ldots, f_{r}(x)$ and where $n$ is an arbitrary positive integer (see \cite{gj}).
Furthermore, the recent proof of V.~Dimitrov \cite{dimi} on the Schinzel-Zassenhaus conjecture now implies the same result,
but with a different constant, unconditionally under rather general conditions on the fixed $f_{j}(x)$.  
In particular, with $f_{r}(x)$ a power of $x$, as in all the examples in Section~\ref{sectapplications}, 
such a result holds unconditionally.

%%%%%%%%%%%

\section{Factorization of polynomials in two variables}\label{twovarsection}
In this section, we are interested in polynomials of the general form
\begin{equation}\label{generalpoly}
F(x,y) = \sum_{j=0}^{r} f_{j}(x) y^{j} \in \mathbb Z[x,y],
\qquad \text{where } r \ge 1 \text{ and } f_{r}(x) \ne 0,
\end{equation}
and determining an irreducibility result for $F(x,y^{n})$ where $n$ is a positive integer. 
Suppose $F(x,y)$ is irreducible in $\mathbb Z[x,y]$ so that, in particular, there is no 
polynomial in $\mathbb Z[x]$ other than $\pm 1$ that divides every $f_{j}(x)$.  
Viewing $F(x,y)$ as a polynomial in $y$, we let $\alpha$ denote a root of $F(x,y) = 0$
(so $F(x,\alpha) = 0$) and consider the field $K = \mathbb Q(x)[\alpha]$.  
An important theorem due to Alfredo Capelli (see \cite{schinzelbook}) implies that 
$F(x,y^{n})$ is reducible in $\mathbb Z[x,y]$ if and only if the polynomial
$y^{n} - \alpha$ is reducible in $K[y]$.  Equally, important, Capelli's work also 
gives us precise knowledge on when $y^{n} - \alpha$ factors in $K[y]$.  He shows
that $y^{n} - \alpha$ factors in $K[y]$ if and only if at least one of the following holds:
\vskip 5pt \noindent
(i) \parbox[t]{15cm}{The number $n$ is divisible by a prime $p$ and $\alpha = \beta^{p}$ for some $\beta \in K$.}
\vskip 5pt \noindent
(ii) \parbox[t]{15cm}{The number $n$ is divisible by $4$ and $\alpha = -4\beta^{4}$ for some $\beta \in K$.}
\vskip 5pt \noindent
Observe that the first result of Capelli's above implies that in the case of (i), 
the polynomial $F(x,y^{p})$ is reducible in $\mathbb Z[x,y]$; and in the case of (ii), 
the polynomial $F(x,y^{4})$ is reducible in $\mathbb Z[x,y]$.  
We will determine a small set $\mathcal P$ of odd primes with the property that
if $\alpha = \beta^{p}$ for some odd prime $p$ and some $\beta \in K$, then $p \in \mathcal P$.  
Thus, determining the $n$ for which $F(x,y^{n})$ is irreducible (or reducible) in $\mathbb Z[x,y]$ will be 
reduced to checking the factorization of $F(x,y^{p})$ for $p \in \mathcal P$ and checking the
factorization of $F(x,y^{4})$.  

A particularly easy consequence of $\alpha = \beta^{p}$ is that 
the norm of $\alpha$ in $K$ over $\mathbb Q(x)$ is a $p$th power of an element of $\mathbb Q(x)$.  
Thus, from \eqref{generalpoly}, we see that if $\alpha = \beta^{p}$, 
then $(-1)^{r} f_{0}(x)/f_{r}(x)$ is a $p$th power in $\mathbb Q(x)$.  
Hence, an easy approach to determining a set $\mathcal P$ of odd primes $p$ for which $F(x,y^{p})$ can factor non-trivially 
is simply to take $\mathcal P$ to be the odd primes $p$ for which the polynomial $(-1)^{r} f_{0}(x)/f_{r}(x)$ is a $p$th power.  
However, one case of interest to us is the case where $(-1)^{r} f_{0}(x)/f_{r}(x) = \pm 1$ which 
is a $p$th power for every odd prime $p$.  With this in mind, an immediate goal is to determine something more, perhaps along similar lines, that one can use to narrow the search for primes $p$ for which 
$F(x,y^{p})$ factors non-trivially.  

We suppose as above that $F(x,y)$ is irreducible in $\mathbb Z[x,y]$. 
We also suppose that at least one of the $f_{j}(x)$ appearing in \eqref{generalpoly} has degree $> 0$
and that the coefficient of $y^{r}$ in $F(x,y)$ is a polynomial in $x$ (possibly of degree $0$) having a positive leading coefficient.  
Suppose further that $\alpha = \beta^{p}$ as described in (i) above.  
Recall that $\beta \in K = \mathbb Q(x)[\alpha]$.  On the other hand, $\alpha = \beta^{p}$ implies
$\alpha \in \mathbb Q(x)[\beta]$.  Since the degree of the extension $\mathbb Q(x)[\alpha]$ over $\mathbb Q(x)$
is $r$, we deduce that the degree of the extension $\mathbb Q(x)[\beta]$ over $\mathbb Q(x)$ 
is $r$.  Hence, there is an irreducible polynomial in $\mathbb Z[x,y]$ of degree $r$ in $y$
having $\beta$ as a root, say
\begin{equation}\label{betapoly}
G(x,y) = \sum_{j=0}^{r} g_{j}(x) y^{j} \in \mathbb Z[x,y],
\qquad \text{where } g_{r}(x) \ne 0.
\end{equation}
As with $F(x,y)$, we note that there is no 
polynomial in $\mathbb Z[x]$ other than $\pm 1$ that divides every $g_{j}(x)$, and we may suppose
that the coefficient of $y^{r}$ in $G(x,y)$ is a polynomial in $x$ (possibly of degree $0$) having a positive leading coefficient.  

Let $\alpha_{1}, \ldots, \alpha_{r}$ be the conjugates of $\alpha$ including $\alpha$ so that
$F(x,\alpha_{j}) = 0$ for $1 \le j \le r$.  
Since $\beta \in K = \mathbb Q(x)[\alpha]$, we can write $\beta$ 
as a polynomial in $\alpha$ with coefficients in $\mathbb Q(x)$, 
say $\beta = H(\alpha) \in \mathbb Q(x)[\alpha]$.     
We denote by $\beta_{j}$ the field conjugate $H(\alpha_{j})$.  
Observe that $\alpha = \beta^{p} = H(\alpha)^{p}$.  
Thus, each $\alpha_{j}$ is a root of $y = H(y)^{p}$; in other words,
$\alpha_{j} = H(\alpha_{j})^{p}$ and, hence, $\alpha_{j} = \beta_{j}^{p}$, for $1 \le j \le r$.  

We begin here with a description of the basic ``idea'' behind our approach connecting $G(x,y)$ to $F(x,y)$ and
leading us to a set of odd primes $\mathcal P$ as indicated above.
If we view $G(x,y)$ as a polynomial in $\mathbb F_{p}(x)[y]$, and consider the $r$ roots
$\beta'_{1}, \ldots, \beta'_{r}$ of $G(x,y)$
in an extension $K'$ of $\mathbb F_{p}(x)$, then 
\[
G(x,y)^{p} = \sum_{j=0}^{r} g_{j}(x)^{p} y^{pj} = g_{r}(x)^{p} \prod_{j=1}^{r} \big( y^{p} - (\beta'_{j}) ^{p} \big).
\]
Consequently, the polynomial $H(x,y) = \sum_{j=0}^{r} g_{j}(x)^{p} y^{j}$ will have roots 
$\big(  \beta'_{1}  \big)^{p}$, $\ldots, \big(  \beta'_{r}  \big)^{p}$ in $K'$.  Thus, in $\mathbb F_{p}[x,y]$, 
the polynomials $F(x,y)$ and $H(x,y)$ will differ at most by a factor in $\mathbb F_{p}[x]$.  
With a little more work, one can see that $F(x,y) = H(x,y)$
in $\mathbb F_{p}[x,y]$.  We give details for how one can establish such a result next.
Of some importance to us, what follows works not just for odd primes $p$ but for $p = 2$ as well.  

Let 
\[
\sigma_{1} = \sum_{j=1}^{r} \alpha_{j}, \ 
\sigma_{2} = \sum_{1 \le i < j \le r} \alpha_{i} \alpha_{j},  \ \ldots, \ \sigma_{r} = \prod_{j=1}^{r} \alpha_{j}, 
\]
denote the elementary symmetric functions of the roots of $F(x,y)$ viewed as a polynomial in $y$ over $\mathbb Q(x)$.  
Thus,
\[
\sigma_{1} = - \dfrac{f_{r-1}(x)}{f_{r}(x)}, \ 
\sigma_{2} = \dfrac{f_{r-2}(x)}{f_{r}(x)}, \  \ldots, \ \sigma_{r} = (-1)^{r} \,\dfrac{f_{0}(x)}{f_{r}(x)}.
\]
Let 
\[
\sigma'_{1} = \sum_{j=1}^{r} \beta_{j}, \ 
\sigma'_{2} = \sum_{1 \le i < j \le r} \beta_{i} \beta_{j},  \ \ldots, \ \sigma'_{r} = \prod_{j=1}^{r} \beta_{j}.
\]
For each $j \in \{ 1, 2, \dots, r \}$, we have $\sigma'_{j} = (-1)^{j} g_{r-j}(x)/g_{r}(x) \in \mathbb Q(x)$.  

We show next that there exist $w_{i,j}(x) \in \mathbb Z[x]$ satisfying
\begin{equation}\label{sigmaconnection}
\sigma_{j} = (\sigma'_{j})^{p} + \dfrac{p}{g_{r}(x)^{p}} \sum_{0 \le i < j} g_{r-i}(x) w_{i,j}(x)
\quad \text{for } 1 \le j \le r.
\end{equation}
Before turning to the proof of \eqref{sigmaconnection}, we explain our interest in it.  
Suppose  then for the moment that we know \eqref{sigmaconnection} holds.  
We rewrite $\sigma'_{j} = (-1)^{j} g_{r-j}(x)/g_{r}(x)$ to deduce that
\[
\sigma_{j} = \dfrac{1}{g_{r}(x)^{p}} \bigg( (-1)^{pj} g_{r-j}(x)^{p} + p \sum_{0 \le i < j} g_{r-i}(x) w_{i,j}(x) \bigg)
\quad \text{for } 1 \le j \le r,
\]
where we include the $p$ in the exponent $pj$ of $-1$ so that the identity also holds for $p = 2$.  
Since the $\sigma_{j}$ represent the elementary symmetric functions in $\alpha_{1}, \dots, \alpha_{r}$, the polynomial
\[
F_{0}(x,y) =  \sum_{j=0}^{r} \bigg( (-1)^{(p+1)j} g_{r-j}(x)^{p} + (-1)^{j} p \sum_{0 \le i < j} g_{r-i}(x) w_{i,j}(x) \bigg) y^{r-j}
\]
has roots $\alpha_{1}, \dots, \alpha_{r}$.  Let $u(x)$ denote an irreducible polynomial in $\mathbb Z[x]$.  
We claim that $u(x)$ cannot divide every coefficient of $F_{0}(x,y)$ as a polynomial in $y$.  Assume otherwise.
Then looking at the coefficients of $y^{r}$, $y^{r-1}$, $\ldots$, $y^{0}$ for $F_{0}(x,y)$ in turn, we deduce
that $u(x)$ divides each of $g_{r}(x)$, $g_{r-1}(x)$, $\ldots$, $g_{0}(x)$.  
This contradicts that no polynomial in $\mathbb Z[x]$ other than $\pm 1$ divides every $g_{j}(x)$.
Hence, $F_{0}(x,y)$ is a polynomial in $y$ with coefficients having 
no common factors in $\mathbb Z[x]$ other than $\pm 1$.
Since $F(x,y)$ also has this property and the coefficients of $y^{r}$ in both $F_{0}(x,y)$ and $F(x,y)$
are polynomials in $x$ having a positive leading coefficient, we obtain
\[
F(x,y) = F_{0}(x,y) \equiv \sum_{j=0}^{r} g_{j}(x)^{p} y^{j} \pmod{p},
\]
which as noted above, in our discussion of $H(x,y)$, was our immediate goal.  

We turn now to the proof of \eqref{sigmaconnection}.
For the proof, we recall that any symmetric polynomial in $\beta_{1}, \dots, \beta_{r}$ with coefficients in $\mathbb Z$ 
can be expressed as a polynomial in $\sigma'_{1}, \dots, \sigma'_{r}$ with coefficients in $\mathbb Z$.  In fact, we make
use of a particular argument for obtaining the polynomial in $\sigma'_{1}, \dots, \sigma'_{r}$.  
For a symmetric polynomial $h(\beta_{1}, \dots, \beta_{r}) \in {\mathbb Z}[\beta_{1}, \dots, \beta_{r}]$,
set $T = T_{h}$ to be the set of $r$-tuples $(\ell_{1},\dots,\ell_{r})$ with the coefficient of 
$\beta_{1}^{\ell_{1}} \cdots \beta_{r}^{\ell_{r}}$ in 
$h(\beta_{1}, \dots, \beta_{r})$ non-zero.  We define the {\it size}
of $h$ to be $(k_{1},\dots,k_{r})$ where $(k_{1},\dots,k_{r})$ is the
element of $T$ with $k_{1}$ as large as possible, $k_{2}$ as large as
possible given $k_{1}$, etc.  Since $h(\beta_{1}, \dots, \beta_{r})$
is symmetric, it follows that $(\ell_{1},\dots,\ell_{r}) \in T$ if and
only if each permutation of $(\ell_{1},\dots,\ell_{r})$ is in $T$.
This implies that $k_{1} \ge k_{2} \ge \cdots \ge k_{r}$.  Observe
that we can use the notion of size to form an ordering on the symmetric polynomials in
${\mathbb Z}[\beta_{1}, \dots, \beta_{r}]$ in the sense that if $h_{1}$
has size $(k_{1},\dots,k_{r})$ and $h_{2}$ has size $(k'_{1},\dots,k'_{r})$,
then $h_{1} > h_{2}$ if there is an $i \in \{ 0,1,\dots,r-1 \}$ 
such that $k_{1}=k'_{1},\dots,k_{i}=k'_{i}$, and $k_{i+1} > k'_{i+1}$.
Note that the elements of ${\mathbb Z}[\beta_{1}, \dots, \beta_{r}]$ which have
size $(0,0,\dots,0)$ are precisely the constants (the elements of ${\mathbb Z}$).

Since $\alpha_{j} = \beta_{j}^{p}$ for $1 \le j \le r$, we may view $\sigma_{j}$ as
a symmetric polynomial in $\beta_{1}, \dots, \beta_{r}$ of size $(p, \dots, p, 0, \dots, 0)$,
where the number of leading $p$'s is $j$.  
In general, we can take a given symmetric polynomial in ${\mathbb Z}[\beta_{1}, \dots, \beta_{r}]$ 
of size $(k_{1},\dots,k_{r})$ and write it in the form
\begin{equation}\label{symmetricreduction}
m \cdot \big(  \sigma'_{1}  \big)^{k_{1}-k_{2}} \big(  \sigma'_{2}  \big)^{k_{2}-k_{3}} 
\cdots \big(  \sigma'_{r-1}  \big)^{k_{r-1}-k_{r}} \big(  \sigma'_{r}  \big)^{k_{r}} 
+ s(\beta_{1}, \ldots, \beta_{r}),
\end{equation}
where $m \in \mathbb Z$ and $s(\beta_{1}, \ldots, \beta_{r})$ is a symmetric polynomial of 
size less than $(k_{1},\dots,k_{r})$.  Here, $m$ is chosen so that when the expression above is
expanded the term corresponding to $m \beta_{1}^{k_{1}} \beta_{2}^{k_{2}} \cdots \beta_{r}^{k_{r}}$ 
agrees with the corresponding term in the given symmetric polynomial.   
We proceed to do this for $\sigma_{j}$ which again is of size $(p, \dots, p, 0, \dots, 0)$.  Thus, 
\[
\sigma_{j} = \big(  \sigma'_{j}  \big)^{p} + s_{1}(\beta_{1}, \ldots, \beta_{r}),
\]
for some symmetric polynomial in ${\mathbb Z}[\beta_{1}, \dots, \beta_{r}]$.  Of some significance here is
the multinomial theorem applied to $\big(  \sigma'_{j}  \big)^{p}$ assures us that each coefficient of
\[
s_{1}(\beta_{1}, \ldots, \beta_{r}) = \sigma_{j} - \big(  \sigma'_{j}  \big)^{p} \in {\mathbb Z}[\beta_{1}, \dots, \beta_{r}]
\]
is divisible by $p$.  Thus, we may write
\[
s_{1}(\beta_{1}, \ldots, \beta_{r}) = p \cdot s_{2}(\beta_{1}, \ldots, \beta_{r}),
\]
for some symmetric $s_{2} = s_{2}(\beta_{1}, \ldots, \beta_{r})  \in {\mathbb Z}[\beta_{1}, \dots, \beta_{r}]$.  
As the size of $s_{2}$ is necessarily less than the size $(p, \dots, p, 0, \dots, 0)$ of $\sigma_{j}$, we deduce
that the size of $s_{2}$ is of the form $(k_{1},\dots,k_{r})$ where $k_{j} < p$.  
Let $i \in \{ 1, 2, \dots, j \}$ be minimal such that $k_{i} < p$.  

Suppose first that $i > 1$.  
Then we rewrite $s_{2}(\beta_{1}, \ldots, \beta_{r})$ using \eqref{symmetricreduction} as
\[
m_{2} \big(  \sigma'_{1}  \big)^{k_{1}-k_{2}} \big(  \sigma'_{2}  \big)^{k_{2}-k_{3}} 
\cdots \big(  \sigma'_{r-1}  \big)^{k_{r-1}-k_{r}} \big(  \sigma'_{r}  \big)^{k_{r}} 
+ s_{3}(\beta_{1}, \ldots, \beta_{r}),
\]
where $m_{2} \in \mathbb Z$ and where $s_{3} = s_{3}(\beta_{1}, \ldots, \beta_{r})$ is a symmetric polynomial 
in ${\mathbb Z}[\beta_{1}, \dots, \beta_{r}]$.  
Since $k_{1} = p$, we note that the first term above has total degree $k_{1} \le p$ in
the variables $\sigma'_{1}, \sigma'_{2}, \ldots, \sigma'_{r}$.  It will be of some significance that 
similarly the total degree in each polynomial in $\sigma'_{1}, \sigma'_{2}, \ldots, \sigma'_{r}$ below is 
$\le p$ as we proceed.  
Observe further that $k_{i-1} = p$ and $k_{i} < p$ implies that $k_{i-1}-k_{i}$ is positive.  
Thus, in the first of the two terms above,  
$\sigma'_{i-1}$ appears as a factor.  
Note also that the size of $s_{3}$ is less than the size of $s_{2}$ and, in particular, 
still less than $(p, \dots, p, 0, \dots, 0)$.  Thus, if the size still has a first component of $p$, 
then we can repeat this process, rewriting $s_{3}$ as a multiple of at least one of 
$\sigma'_{t}$ for $1 \le t \le i-1$ plus a symmetric polynomial $s_{4}$ of smaller size, 
and then rewriting $s_{4}$ as well if the first component of the size remains $p$, and so on.  
As the sizes are decreasing, at some point the summand involving a new symmetric polynomial 
will have size $(k'_{1},\dots,k'_{r})$ with $k'_{1} < p$.  

Combining terms arising from above, we see that 
\[
\sigma_{j} = \big(  \sigma'_{j}  \big)^{p} + p \sum_{t=1}^{i-1} \sigma'_{t} s'_{t}(\sigma'_{1}, \ldots, \sigma'_{r})
+ p \,s_{0}(\beta_{1}, \ldots, \beta_{r}),
\]
where each $s'_{t} = s'_{t}(\sigma'_{1}, \ldots, \sigma'_{r})$ is a polynomial in ${\mathbb Z}[\sigma'_{1}, \dots, \sigma'_{r}]$
and $s_{0} = s_{0}(\beta_{1}, \ldots, \beta_{r})$ is a symmetric polynomial in ${\mathbb Z}[\beta_{1}, \dots, \beta_{r}]$
of size $(k'_{1},\dots,k'_{r})$ with $k'_{1} < p$. 
Note that, for $1\le t \le i-1$, the total degree of $\sigma'_{t} s'_{t}$ in the variables $\sigma'_{1}, \dots, \sigma'_{r}$ is $\le p$, so consequently $s'_{t}$ has total degree $< p$.  
At this point, we simply repeat the process of using and reusing \eqref{symmetricreduction} to rewrite 
$s_{0}$ as a polynomial in ${\mathbb Z}[\sigma'_{1}, \dots, \sigma'_{r}]$.  The significance of having
$k'_{1} < p$ is that at each stage from this point on, the size of the symmetric polynomials introduced will
have size with first component $< p$.  Observe that if $k_{1} < p$ in the first term
\[
m \big(  \sigma'_{1}  \big)^{k_{1}-k_{2}} \big(  \sigma'_{2}  \big)^{k_{2}-k_{3}} 
\cdots \big(  \sigma'_{r-1}  \big)^{k_{r-1}-k_{r}} \big(  \sigma'_{r}  \big)^{k_{r}} 
\]
of \eqref{symmetricreduction}, then the total degree of this term as a polynomial in 
$\sigma'_{1}, \dots, \sigma'_{r}$ is $k_{1} < p$.  In other words, we can repeatedly 
use \eqref{symmetricreduction} to rewrite 
$s_{0}$ as a polynomial $s'_{0} = s'_{0}(\sigma'_{1}, \ldots, \sigma'_{r})$ 
in the variables $\sigma'_{1}, \dots, \sigma'_{r}$ with total degree $< p$.  

Now, we have
\begin{equation}\label{sigmatosigmaprimes}
\sigma_{j} = \big(  \sigma'_{j}  \big)^{p} + p \sum_{t=1}^{i-1} \sigma'_{t} s'_{t}(\sigma'_{1}, \ldots, \sigma'_{r})
+ p \,s'_{0}(\sigma'_{1}, \ldots, \sigma'_{r}),
\end{equation}
where each $s'_{t}$ for $0 \le t \le i-1$ is a polynomial in ${\mathbb Z}[\sigma'_{1}, \dots, \sigma'_{r}]$ 
with total degree $< p$ in the variables $\sigma'_{1}, \dots, \sigma'_{r}$.  
We make the substitution $\sigma'_{t} = (-1)^{t} g_{r-t}(x)/g_{r}(x)$ in \eqref{sigmatosigmaprimes} to obtain
\[
\sigma_{j} = \big(  \sigma'_{j}  \big)^{p} + p \sum_{t=1}^{i-1} \dfrac{g_{r-t}(x) w_{t,j}(x)}{g_{r}(x)^{p}} 
+ \dfrac{p \,w_{0,j}(x)}{g_{r}(x)^{p-1}},
\]
where each $w_{t,j}(x)$ for $0 \le t \le i-1$ is in $\mathbb Z[x]$.  Since $i \le j$, we deduce that \eqref{sigmaconnection}
holds.  As noted earlier, this leads to 
\begin{equation}\label{pthpowerequatforF}
F(x,y) \equiv \sum_{j=0}^{r} g_{j}(x)^{p} y^{j} \pmod{p}.
\end{equation}
The above follows from the condition $\alpha = \beta^{p}$ in (i) and holds for any prime $p$ satisfying this condition.

Recall that we have required that at least one of the $f_{j}(x)$ appearing in \eqref{generalpoly} has degree $> 0$. 
We also took the coefficient of $y^{r}$ in $F(x,y)$ to be a polynomial in $x$ (possibly of degree $0$) having a positive leading coefficient.  This latter condition can be removed at this point since what we are interested in is the information from
\eqref{pthpowerequatforF} that each coefficient of $F(x,y)$ is a $p$th power modulo $p$, and if necessary we can multiply both sides of \eqref{pthpowerequatforF} by $(-1)^{p}$ to obtain a result in the case that the latter condition does not hold.
We define the height of $f_{j}(x)$ to be the maximum of the absolute values of the coefficients of $f_{j}(x)$. 
Suppose $p$ is an odd prime that is greater than the height $H(f_{j})$ for each $f_{j}(x)$.  This is more than we will need and only really require here that $f_{j}(x)$ is not constant modulo $p$ for some $j$.  
Thus, for some $j$, the coefficient $f_{j}(x)$ of $y^{j}$ on the left-hand side of
\eqref{pthpowerequatforF} viewed modulo $p$ is of degree $> 0$.  
It follows that the corresponding $g_{j}(x)$ is a non-constant polynomial modulo $p$ on the right-hand side of \eqref{pthpowerequatforF}.  In this case, the polynomial
$g_{j}(x)^{p}$ modulo $p$ must be of degree at least $p$ which in turn implies that $\deg f_{j} \ge p$.  Thus, we may take
$\mathcal P$ to be the set of odd primes $p$ for which 
\begin{equation}\label{primebound}
p \le \max\{ \deg f_{0}, \dots, \deg f_{r}, H(f_{0}), \dots, H(f_{r}) \}.
\end{equation}
In other words, \eqref{pthpowerequatforF} implies \eqref{primebound}.  
In particular, with $\mathcal P$ so chosen, if $F(x,y^{p})$ for $p \in \mathcal P$ and $F(x,y^{4})$ are all irreducible, then $F(x,y^{n})$ is irreducible for every positive integer $n$.   Recalling the opening paragraph of this section, we summarize the above as follows. 

\begin{theorem}\label{thmgeneralpoly}
Let $F(x,y)$ be given as in \eqref{generalpoly}, with $F(x,y)$ irreducible in $\mathbb Z[x,y]$ and with at least one $f_{j}(x)$ non-constant.  
Let $\mathcal P$ be the set of odd primes $p$ satisfying \eqref{primebound}.  
Then $F(x,y^{n})$ is reducible in $\mathbb Z[x,y]$ if and only if either $p|n$ and $F(x,y^{p})$ is reducible for some $p \in \mathcal P$, 
or $2|n$ and $F(x,y^{2})$ is reducible, or $4|n$ and $F(x,y^{4})$ is reducible.   
In particular, if $F(x,y^{p})$ for $p \in \mathcal P$ and $F(x,y^{4})$ are all irreducible, then $F(x,y^{n})$ is irreducible in $\mathbb Z[x,y]$ 
for every positive integer $n$.  
Furthermore, if $p \in \mathcal P \cup \{ 2 \}$ and $F(x,y^{p})$ is reducible, then $F(x,y)$ is necessarily of the form \eqref{pthpowerequatforF}.  
\end{theorem}

There is some value in looking closer at the case that $\alpha = -4\beta^{4}$ as in (ii), which we briefly sketch here.  
We suppose now that $\alpha = -4\beta^{4}$ and $G(x,y)$ in \eqref{betapoly} has $\beta$ as a root, where again
$G(x,y)$ is an irreducible polynomial in $\mathbb Z[x,y]$ and $g_{r}(x)$ has a positive leading coefficient.
For $1 \le j \le r$, we define $\sigma_{j}$ and $\sigma'_{j}$ as before.  Instead of \eqref{sigmaconnection}, 
we have in this case
\begin{equation}\label{sigmaconnectionnew}
\sigma_{j} = (-4)^{j} (\sigma'_{j})^{4} + \dfrac{2^{2j+1}}{g_{r}(x)^{4}} \sum_{0 \le i < j} g_{r-i}(x) w_{i,j}(x)
\quad \text{for } 1 \le j \le r,
\end{equation}
for some polynomials $w_{i,j}(x) \in \mathbb Z[x]$.  The proof of \eqref{sigmaconnectionnew} is analogous to the
proof of \eqref{sigmaconnection}.  The implications are a bit different, and we discuss that next.

We will want to take advantage of a little more information on the $w_{i,j}(x)$.  Observe that
\[
\sigma_{r} 
= \alpha_{1} \cdots \alpha_{r}
= (-4)^{r} \big( \beta_{1}^{4} \cdots \beta_{r}^{4} \big)
= (-4)^{r} (\sigma'_{r})^{4}.
\]
Thus, in regards to \eqref{sigmaconnectionnew}, we can take $w_{i,r}(x) = 0$ for $0 \le i < r$.

Recalling $\sigma'_{j} = (-1)^{j} g_{r-j}(x)/g_{r}(x)$, we deduce
\[
F_{1}(x,y) =  \sum_{j=0}^{r} \bigg( 2^{2j} g_{r-j}(x)^{4} + (-1)^{j} 2^{2j+1} \sum_{0 \le i < j} g_{r-i}(x) w_{i,j}(x) \bigg) y^{r-j}
\]
has roots $\alpha_{1}, \ldots, \alpha_{r}$.  As with $F_{0}(x,y)$, we would like to determine whether the coefficients of 
$F_{1}(x,y)$ viewed as a polynomial in $y$ have any common factors.  Analogous to $F_{0}(x,y)$, we see that every
coefficient of $F_{1}(x,y)$ cannot be divisible by an irreducible polynomial $u(x)$ in $\mathbb Z[x]$ if either $u(x)$
is non-constant or $u(x)$ is an odd prime.  In other words, the only irreducible polynomials in $\mathbb Z[x]$
that might divide every coefficient of $F_{1}(x,y)$ are $\pm 2$.  Suppose $2^{s}$ divides every coefficient of 
$F_{1}(x,y)$ for some positive integer $s$ with $s$ maximal.  It follows that $F_{1}(x,y)/2^{s}$ is an irreducible
polynomial in $\mathbb Z[x,y]$ having roots $\alpha_{1}, \ldots, \alpha_{r}$.  
Since $g_{r}(x)$ and $f_{r}(x)$ have a positive leading coefficient, we necessarily have
\begin{align*}
F(x,y) &= \dfrac{F_{1}(x,y)}{2^{s}} \\[5pt]
&= \dfrac{1}{2^{s}} \sum_{j=0}^{r} \bigg( 2^{2j} g_{r-j}(x)^{4} + (-1)^{j} 2^{2j+1} \sum_{0 \le i < j} g_{r-i}(x) w_{i,j}(x) \bigg) y^{r-j}.
\end{align*}

Let $e_{j}$ be the non-negative integer for which
$2^{e_{j}}$ is the largest power of $2$ dividing the coefficient of $y^{j}$ in $F_{1}(x,y)$ in $\mathbb Z[x]$. 
Let $e'_{j}$ be the non-negative integer for which $2^{e'_{j}}$ is the largest power of $2$ dividing $g_{j}(x)$ in $\mathbb Z[x]$.  
In particular, $e_{r} = 4 e'_{r}$.  
From the formulation of $F_{1}(x,y)$ above, we see that 
$e_{j} \ge 2r -2j$ for each $j \in \{ 0, 1, \ldots, r \}$ and
if $e'_{j} = 0$, then $e_{j} = 2r -2j$.  Also, the irreducibility of $G(x,y)$ guarantees that some $e'_{i} = 0$.  
Given such an $i$, we deduce $s \le 2r-2i$, and $2^{2(i-j)}$ divides the coefficient of $y^{j}$ in $F(x,y)$ for $0 \le j < i$. 
In particular, if the maximal such $i$ is $\ge 1$, then $f_{0}(x)$ is divisible by $4$.  
Suppose the only such $i$ is $i = 0$.  Then either $s < 2r$ and $f_{0}(x)$ is divisible by $2$
or $s = 2r$.  We consider the case that $s = 2r$.   
From $s = 2r$, $e'_{0} = 0$ and $w_{i,r}(x) = 0$ for $0 \le i < r$, we have $f_{0}(x) = g_{0}(x)^{4}$.  
If $r$ is odd, then $s \equiv 2 \pmod{4}$.  As $e_{r} = 4 e'_{r} \equiv 0 \pmod{4}$, we
see that $f_{r}(x)$ is divisible by $4$ in this case.   

Thus far we have seen that either $f_{0}(x)$ is divisible by $2$ or $f_{r}(x)$ is divisible by $4$ unless
$r$ is even, $s = 2r$, $e'_{0} = 0$, and $e'_{j} \ge 1$ for $1 \le j \le r$.  Furthermore, in this case, $s = 2r$ implies
that $2^{2r}$ divides each coefficient of $y^{j}$ in $F_{1}(x,y)$ and, in particular, $2^{2r}$ divides $g_{r}(x)^{4}$.  
Thus, $e'_{r} \ge r/2$. 
We will be interested in the particular case that $r = 4$, so suppose 
$f_{0}(x)$ is not divisible by $2$ and $f_{r}(x)$ is not divisible by $4$, giving us the above conclusions with $r = 4$.  
We make use of the explicit formula in this case that
\begin{align*}
\sigma_{1} &= \sum_{j=1}^{4} (-4 \beta_{j}^{4}) \\[5pt]
&= -4 \big(
(\sigma'_{1})^{4} - 4 (\sigma'_{1})^{2} (\sigma'_{2}) + 4 (\sigma'_{1}) (\sigma'_{3}) + 2 (\sigma'_{2})^{2} - 4 (\sigma'_{4})
\big),
\end{align*}
which leads to the more precise formulation of the polynomial
\begin{gather*}
4 g_{3}(x)^{4} - 16 g_{3}(x)^{2} g_{2}(x) g_{4}(x) + 16 g_{3}(x) g_{1}(x) g_{4}(x)^{2} \\
+ 8 g_{2}(x)^{2} g_{4}(x)^{2} - 16 g_{0}(x) g_{4}(x)^{3}
\end{gather*}
for the coefficient of $y^{r-1} = y^{3}$ in $F_{1}(x,y)$.  This polynomial is required to be divisible by $2^{2r} = 2^{8}$ in $\mathbb Z[x]$.  
As $e'_{4} \ge 4/2 = 2$ and $e'_{j} \ge 1$ for $1 \le j \le 3$, we see that each term beyond the first term is divisible by $2^{9}$
in $\mathbb Z[x]$, so the first term $4 g_{3}(x)^{4}$ must be divisible by $2^{8}$.  It follows that $e'_{3} \ge 2$.  We deduce that
the coefficient of $y^{3}$ in $F_{1}(x,y)$ is divisible by $2^{9}$ so that the coefficient of $y^{3}$ in $F(x,y)$ is divisible by $2$.  
Hence, we have that in the case that $\alpha = -4\beta^{4}$ as in (ii) and $r = 4$, then either $f_{0}(x)$ or $f_{3}(x)$ is divisible 
by $2$ or $f_{4}(x)$ is divisible by $4$.

To summarize, in the case $r = 4$, we have shown that if $F(x,y^{4})$ is reducible, then either $F(x,y^{2})$ is reducible and 
\eqref{pthpowerequatforF} holds or one of $f_{0}(x)$ or $f_{3}(x)$ is divisible by $2$ or $f_{4}(x)$ is divisible by $4$.
For general odd $r$, we have also seen that if $F(x,y^{4})$ is reducible, then either $F(x,y^{2})$ is reducible and 
\eqref{pthpowerequatforF} holds or $f_{0}(x)$ is divisible by $2$ or $f_{4}(x)$ is divisible by $4$.  

We end this section with an example that will play a role later in our paper.
\vskip 8pt \noindent
\textbf{Example.} 
Let 
\[
F(x,y) = 1-x(x+1)y-2x^2y^2-x^2(x+1)y^3+x^4y^4
\]
so that $F(x,y)$ is of the form \eqref{generalpoly} with 
$f_{0}(x) = 1$, $f_{1}(x) = -x(x+1)$, $f_{2}(x) = -2x^{2}$, $f_{3}(x) = -x^{2}(x+1)$ and $f_{4}(x) = x^{4}$.
One checks directly that $F(x,y)$ and $F(x,y^{4})$ are irreducible.  Since $f_{1}(x) = -x(x+1)$, we see that for every odd prime $p$,
the congruence in \eqref{pthpowerequatforF} does not hold.  
Theorem~\ref{thmgeneralpoly} now implies that $F(x,y^{n})$ is irreducible for every positive integer $n$.

%%%%%%%%%%%%%%%%%%%%

\section{A theorem connecting single and multivariate polynomials}\label{prelimsect}
To help with the statements of the results below and the proofs to follow, we discuss
notation here.  The expression $a \mmod k$ will denote the unique integer 
$b$ in $[0,k)$ for which $a \equiv b \pmod{k}$.  If $u$ is a real number,
$\lfloor u \rfloor$ will denote the greatest integer $\le u$, 
$\lceil u \rceil$ the least integer $\ge u$, and 
$\Vert u\Vert$ the 
minimal distance from $u$ to an integer.  We will use 
$\{ u \}$ to denote $u - \lfloor u \rfloor$ unless it is clear from the context
that $\{ u \}$ refers to a set consisting of the single element $u$.  
In the case that $f(x)$ is a polynomial, $\Vert f\Vert$ represents the $2$-norm
defined as the squareroot of the sum of the squares of the coefficients of $f(x)$.  
As before, the reciprocal of the polynomial $f(x)$ is $\tilde{f}(x) = x^{\deg f} f(1/x)$,
and the non-reciprocal part of $f(x) \in \mathbb Z[x]$ is $f(x)$ removed of each
irreducible reciprocal factor in $\mathbb Z[x]$ with a positive leading coefficient.

\begin{theorem}\label{mainlemma}
Let
$F(x) = \sum_{j=0}^{r} a_{j} x^{d_{j}} \in \mathbb Z[x]$, where
$0 = d_{0} < d_{1} < \cdots < d_{r}$ and $a_{r} a_{0} \ne 0$.   
Let $\varepsilon \in (0,1/4]$, and let $k_{0}$ be a real number $\ge 2$.  
Set  $\kappa = \lceil 1/\varepsilon \rceil$.  
Suppose that
\[
\deg F \ge \max \bigg\{ \dfrac{\kappa^{2N-2}}{2} \bigg( \varepsilon - \dfrac{1}{\kappa}  \bigg)^{-1}, 
k_{0} \big( \kappa^{N-1} + \varepsilon \big) \bigg\}, 
\ \text{ where }  N = 2\,\Vert F \Vert^{2} + 2 r - 5.
\]
If the non-reciprocal part of $F(x)$ is reducible in $\mathbb Z[x]$, 
then there exists an integer $k$ in the interval $[k_{0}, (\deg F)/(1-\varepsilon))$ satisfying:
\begin{itemize}
\item[(i)] 
For each $j \in \{0,1,\dots,r\}$, the number 
$d_{j} \mmod k$ is in 
\[
[0,\varepsilon k) \cup ((1-\varepsilon) k,k).
\]
\item[(ii)]
If $\overline{d}_{j}$ and $\ell_{j}$ are defined by 
\[
\overline{d}_{j} = (d_{j} + \lfloor \varepsilon k \rfloor) \mmod k \quad \text{and} \quad
d_{j} + \lfloor \varepsilon k \rfloor = k \ell_{j} + \overline{d}_{j}
\]
and $G(x,y) = \sum_{j=0}^{r} a_{j} x^{\overline{d}_{j}} y^{\ell_{j}}$,  
then $x^{-M} G(x,y)$ is reducible in $\mathbb Z[x,y]$, where $M$ is a non-negative 
integer chosen as large as possible with the constraint that 
$x^{-M} G(x,y) \in \mathbb Z[x,y]$.
\end{itemize}
\end{theorem}

The above result is very similar to Theorem 2 in \cite{FFK} and follows along very similar lines.  
As the result nevertheless is different, we present the details below.

Before beginning the proof, we observe that $k < d_{r}/(1-\varepsilon)$ implies
\[
k - 1 - \lfloor \varepsilon k \rfloor < k (1 - \varepsilon) < d_{r}.
\]
Hence, $d_{r} + \lfloor \varepsilon k \rfloor \ge k$ so that $\ell_{r} \ge 1$.  Thus, $G(x,y)$ depends on $y$.    

To establish Theorem \ref{mainlemma}, we make use of a seemingly
unconnected combinatorial result of interest in itself.  
Fix $\varepsilon > 0$.  
Suppose that $\a_{1}, \a_{2}, \dots, \a_{\r}$ are distinct non-negative integers
written in increasing order and that we wish to determine an integer
$k \ge 2$ satisfying the property that for each $j \in \{1,2,\dots,\r\}$, we have 
$\a_{j} \equiv  v'_{j} \pmod{k}$ for some $v'_{j} \in (-\varepsilon k,\varepsilon k)$.
In \cite{FFK}, the case $\varepsilon = 1/4$ was of particular importance and
focused on there.      
Observe that the value $k = \lfloor \a_{\r}/\varepsilon \rfloor + 1$ satisfies the property 
stated above for $k$, but we will want a smaller value of $k$.  We show that for each $\r$,
there exists a $V(\r) = V(\r,\varepsilon)$ such that if $\a_{\r} \ge V(\r)$, then there is a 
$k \in [2,\a_{\r}/(1-\varepsilon))$ such that for each $j \in \{1,2,\dots,\r\}$, we have
$\a_{j} \equiv v'_{j} \pmod{k}$ for some $v'_{j} \in (-\varepsilon k,\varepsilon k)$.  

\begin{lemma}\label{ffklemchapt10}
Let $\varepsilon \in (0,1)$, let $\r$ be a positive integer, and let
$k_{0}$ be a real number $\ge 2$.  Set  $\kappa = \lfloor 1/\varepsilon \rfloor + 1$ and
\[
V(\r) = \max \bigg\{ \dfrac{\kappa^{2\r-2}}{2} \bigg( \varepsilon - \dfrac{1}{\kappa}  \bigg)^{-1}, 
k_{0} \big( \kappa^{\r-1} + \varepsilon \big) \bigg\}.
\]  
Let $\a_{1}, \a_{2}, \dots, \a_{\r}$ be non-negative integers satisfying
$\a_{1} < \a_{2} < \cdots < \a_{\r}$ and $\a_{\r} \ge V(\r)$.  
Then there exists an integer $k \in [k_{0},\a_{\r}/(1-\varepsilon))$ such that 
$\a_{j} \mmod k$ is in $[0,\varepsilon k) \cup ((1-\varepsilon) k,k)$
for each $j \in \{1,2,\dots,\r\}$.
\end{lemma}

\begin{proof}
For $\r = 1$, the conditions imply $\a_{\r} \ge k_{0}$ and we can take $k = \a_{\r}$.  
We consider now the case that $\r > 1$.  
Let $D$ be an integer with
\[
1 \le D \le \dfrac{\sqrt{2 \a_{\r}}}{\kappa^{\r-1}} \,\bigg( \varepsilon - \dfrac{1}{\kappa} \bigg)^{1/2}.
\]
The condition $\a_{\r} \ge V(\r)$ implies that such an integer $D$ exists.  
We will show that one may take 
\begin{equation}\label{ffkpap2}
k \in \bigg( \dfrac{\a_{\r}}{\kappa^{\r-1}D + \varepsilon}, \dfrac{\a_{\r}}{D - \varepsilon} \bigg).
\end{equation}
In particular, the lemma will follow by choosing $D = 1$.
Note that the conditions in the lemma, $D = 1$ and $\r > 1$ imply that the interval in \eqref{ffkpap2} has length $> 1$, 
so such a $k$ exists.

Let $x_{j} = \a_{j}/\a_{\r}$ for $j \in \{1,2,\dots,\r\}$.  We want to show that
there is an integer $k$ satisfying \eqref{ffkpap2} and integers $\c_{1}, \c_{2}, \dots, \c_{\r}$ 
such that
\begin{equation}\label{ffkpap3}
|\a_{j} - \c_{j} k| < \varepsilon k \qquad \text{ for } 1 \le j \le \r. 
\end{equation}

We explain next how the
Dirichlet box principle implies that there is an integer $d$ satisfying 
$D \le d \le \kappa^{\r-1}D$, $D|d$ and 
\begin{equation}\label{ffkpap4}
\Vert d x_{j} \Vert \le \dfrac{1}{\kappa} \qquad \text{ for } 1 \le j \le \r-1. 
\end{equation}
For each $d' \in \{ D, 2D, 3D, \dots, (\kappa^{\r-1}+1)D \}$, we define the
point 
\[
P(d') = \big( \{ d'x_{1} \}  , \{ d'x_{2} \}, \dots, \{ d'x_{\r-1} \} \big).
\] 
Each such point has its coordinates in the interval $[0,1)$.  
For each of the $\kappa^{\r-1}$ choices of $u_{j} \in \{ 0, 1, \ldots, \kappa -1 \}$, where $1 \le j \le \r-1$, 
we consider the cube
\begin{align*}
\mathcal C(u_{1}, \dots, u_{\r-1}) = 
\{ (t_{1}, t_{2}, \dots, t_{\r-1}) : u_{j}/\kappa \le t_{j} < &(u_{j}+1)/\kappa \\ &\text{ for } 1 \le j \le \r-1  \}.
\end{align*}
The Dirichlet box principle
implies that there are $d_{1}'$ and $d_{2}'$ in 
\[
\{ D, 2D, \ldots, (\kappa^{\r-1}+1)D \}
\]
with $d_{1}' > d_{2}'$
such that the points $P(d_{1}')$ and $P(d_{2}')$ are both in the same cube
$C(u_{1}, \dots, u_{\r-1})$.  Set 
\[
d = d_{1}'-d_{2}' \in \{ D, 2D, \ldots, \kappa^{\r-1}D \}.
\]
For every $j \in \{ 1, 2, \dots, \r-1 \}$, we obtain
\[
d x_{j} = d_{1}'x_{j} - d_{2}'x_{j} 
= \lfloor d_{1}'x_{j} \rfloor - \lfloor d_{2}'x_{j} \rfloor 
+ \{ d_{1}'x_{j} \} - \{ d_{2}'x_{j} \}.
\]
It follows that each such $d x_{j}$ is within 
\[
|\{ d_{1}'x_{j} \} - \{ d_{2}'x_{j} \}| \le \dfrac{1}{\kappa}
\]
of the integer $\lfloor d_{1}'x_{j} \rfloor - \lfloor d_{2}'x_{j} \rfloor$.  
This establishes \eqref{ffkpap4}.  
Clearly, \eqref{ffkpap4} holds with $j = \r$ as well.  
Observe that the upper bound on $D$ above implies that 
\[
d \le \sqrt{2 \a_{\r}} \,\bigg( \varepsilon - \dfrac{1}{\kappa} \bigg)^{1/2}.
\]
For $1 \le j \le \r$, take $\c_{j}$ to be the nearest
integer to $d x_{j}$.  

For the moment, suppose $\c_{j} \ne 0$ (so that $\c_{j} \ge 1$) for each 
$j \in \{1,2,\dots,\r\}$.  Then \eqref{ffkpap3} follows provided
\[
\dfrac{k}{\a_{\r}} \in \bigcap_{1 \le j \le \r} 
\bigg( \dfrac{x_{j}}{\c_{j} + \varepsilon}, \dfrac{x_{j}}{\c_{j} - \varepsilon} \bigg).
\]
For each $j \in \{1,2,\dots,\r\}$, since $\c_{j} \le d$, we deduce from \eqref{ffkpap4} that
\[
\dfrac{d+(1/\kappa)}{d+\varepsilon} \ge \dfrac{\c_{j}+(1/\kappa)}{\c_{j}+\varepsilon} 
\ge \dfrac{d x_{j}}{\c_{j}+\varepsilon}
\]
and
\[
\dfrac{d-(1/\kappa)}{d-\varepsilon} \le \dfrac{\c_{j}-(1/\kappa)}{\c_{j}-\varepsilon} 
\le \dfrac{d x_{j}}{\c_{j}-\varepsilon}.
\]
Hence, \eqref{ffkpap3} holds provided
\begin{equation}\label{ffkpap5}
\dfrac{k}{\a_{\r}} \in 
\bigg( \dfrac{d+(1/\kappa)}{d(d+\varepsilon)}, \dfrac{d-(1/\kappa)}{d(d-\varepsilon)} \bigg).
\end{equation}
The length of the interval on the right is 
\[
\dfrac{2}{(d^{2}-\varepsilon^{2})} \bigg( \varepsilon - \dfrac{1}{\kappa} \bigg) 
> \dfrac{2}{d^{2}} \bigg( \varepsilon - \dfrac{1}{\kappa} \bigg) \ge \dfrac{1}{\a_{\r}}
\]
so that $k$ exists satisfying \eqref{ffkpap5}.  Observe that \eqref{ffkpap5} and the definition of $d$
imply \eqref{ffkpap2} holds. 

Now, suppose some $\c_{j} = 0$ with $j \in \{1,2,\dots,\r\}$.  We again choose
$k$ so that \eqref{ffkpap5} holds.  For each $\c_{j} \ne 0$, the above argument gives 
$|\a_{j}-\c_{j}k| < \varepsilon k$ as in \eqref{ffkpap3}.  On the other hand, if $\c_{j} = 0$, then
\eqref{ffkpap4} and the definitions of $\c_{j}$, $x_{j}$, and $k$ imply
\[
\kappa d \a_{j} \le \a_{\r} < k 
\dfrac{d \big( d + \varepsilon \big)}{d + (1/\kappa)}
\le \kappa \varepsilon d k.
\]
Hence, \eqref{ffkpap3} holds for such $\c_{j}$ as well, completing the proof.
\end{proof}

\begin{proof}[Proof of Theorem~\ref{mainlemma}.]
Suppose that the non-reciprocal part of $F(x)$ is reducible.  
As described in detail in the proof of Theorem 2 of \cite{FFK}, there exist non-reciprocal polynomials 
$u(x)$ and $v(x)$ in $\mathbb Z[x]$ such that $F(x) = u(x) v(x)$.
Let $W(x) = u(x) \tilde{v}(x)$.    
The polynomials $F(x)$, $\widetilde{F}(x)$, $W(x)$, and $\widetilde{W}(x)$
are distinct polynomials of degree $d_{r}$ with any two having greatest
common divisor of degree $< d_{r}$.  
Note that
\begin{equation}\label{ffkpap8}
F(x) \widetilde{F}(x) = u(x) v(x) \tilde{u}(x) \tilde{v}(x) = W(x) \widetilde{W}(x).
\end{equation}
By comparing the coefficient of $x^{d_{r}}$ on the left and right sides of \eqref{ffkpap8}, 
we deduce that $\Vert F \Vert^{2} = \Vert W \Vert^{2}$. 

We write
$W(x)$ in the form $W(x) = \sum_{j=0}^{s} b_{j} x^{e_{j}}$ where 
the $b_{j}$ are non-zero and $0 = e_{0} < e_{1} < \cdots < e_{s} = d_{r}$.
Then $\Vert W \Vert^{2} = \Vert F \Vert^{2}$ implies 
$s \le \Vert F \Vert^{2} - 1$.  
Consider the set 
\begin{align*}
T = \{ d_{1}, d_{2}, &\dots, d_{r} \} 
\cup \{ d_{r} - d_{1}, d_{r} - d_{2}, \dots, d_{r} - d_{r-1} \} \\
&\cup \{ e_{1}, e_{2}, \dots, e_{s-1} \} 
\cup \{ e_{s} - e_{1}, e_{s} - e_{2}, \dots, e_{s} - e_{s-1} \}.
\end{align*}
Observe that $|T| \le 2\,\Vert F \Vert^{2} + 2 r - 5$.  
We use the lower bound on $d_{r} = \deg F$ in the statement of the
theorem together with Lemma \ref{ffklemchapt10} to deduce that
there exists an integer $k \in [k_{0},d_{r}/(1-\varepsilon))$ such that 
$t \mmod k$ is in $[0, \varepsilon k) \cup ((1-\varepsilon) k,k)$ for every $t \in T$.
Fix such an integer $k$.  
 
Define $\overline{d}_{j}$ and $\ell_{j}$ as in the theorem, and define
$\overline{e}_{j}$ and $m_{j}$, for $0 \le j \le s$, similarly by
\[
\overline{e}_{j} = (e_{j} + \lfloor \varepsilon k \rfloor) \mmod k \qquad \text{and} \qquad
e_{j} + \lfloor \varepsilon k \rfloor = k m_{j} + \overline{e}_{j}.
\]
For $0 \le j \le r$, we also define
$\overline{d}'_{j}$ and $\ell'_{j}$ by
\[
\overline{d}'_{j} = (d_{r}-d_{j} + \lfloor \varepsilon k \rfloor) \mmod k, \qquad \text{and} \qquad
d_{r}-d_{j} + \lfloor \varepsilon k \rfloor = k \ell'_{j} + \overline{d}'_{j};
\]
and, for $0 \le j \le s$, we also define 
$\overline{e}'_{j}$ and  $m'_{j}$ by 
\[
\overline{e}'_{j} = (e_{s}-e_{j} + \lfloor \varepsilon k \rfloor) \mmod k \qquad \text{and} \qquad
e_{s}-e_{j} + \lfloor \varepsilon k \rfloor = k m'_{j} + \overline{e}'_{j}.
\]
Since $t \mmod k$ is in $[0,\varepsilon k) \cup ((1-\varepsilon) k,k)$ for every $t \in T$, we see that
the numbers $\overline{d}_{j}$ and $\overline{d}'_{j}$, for $0 \le j \le r$,
and the numbers $\overline{e}_{j}$ and $\overline{e}'_{j}$, for $0 \le j \le s$,
all lie in the interval $[0, 2\varepsilon k)$. 
Define $G_{1}(x,y) = G(x,y)$ (as in the statement of the theorem), 
\[
G_{2}(x,y) = \sum_{j=0}^{r} a_{j} x^{\overline{d}'_{j}} y^{\ell'_{j}},
\qquad H_{1}(x,y) = \sum_{j=0}^{s} b_{j} x^{\overline{e}_{j}} y^{m_{j}},
\]
and
\[
H_{2}(x,y) = \sum_{j=0}^{s} b_{j} x^{\overline{e}'_{j}} y^{m'_{j}}.
\]
Hence,
\begin{gather}\label{ffkpap10}
\begin{split}
G_{1}(x,x^{k}) = x^{\lfloor \varepsilon k \rfloor} F(x), \qquad
G_{2}(x,x^{k}) = x^{\lfloor \varepsilon k \rfloor} \widetilde{F}(x), \phantom{ssld} \\[5pt]
H_{1}(x,x^{k}) = x^{\lfloor \varepsilon k \rfloor} W(x) \qquad \text{and} \qquad 
H_{2}(x,x^{k}) = x^{\lfloor \varepsilon k \rfloor} \widetilde{W}(x). 
\end{split}
\end{gather} 

Corresponding to \eqref{ffkpap8}, we establish next that
\begin{equation}\label{ffkpap9}
G_{1}(x,y) G_{2}(x,y) = H_{1}(x,y) H_{2}(x,y).
\end{equation}
Recall that we are working with the condition that $\varepsilon \le 1/4$ so that $2\varepsilon k \le k/2$. 
Expanding the left-hand side of \eqref{ffkpap9}, we obtain
an expression of the form $\sum_{j=0}^{J} g_{j}(x) y^{j}$ where possibly
some $g_{j}(x)$ are $0$ but otherwise $\deg g_{j} < k$ for each $j$.  
Since
\[
x^{2\lfloor \varepsilon k \rfloor} F(x) \widetilde{F}(x) = G_{1}(x,x^{k}) G_{2}(x,x^{k})
= \sum_{j=0}^{J} g_{j}(x) x^{kj},
\]
we see that the terms in each $g_{j}(x) x^{kj}$ correspond precisely to the 
terms in the expansion of $x^{2\lfloor \varepsilon k \rfloor} F(x) \widetilde{F}(x)$ having degrees 
in the interval $[kj,k(j+1))$.
Similarly, writing the right-hand side of \eqref{ffkpap9} in the form
$\sum_{j=0}^{J'} h_{j}(x) y^{j}$, we get each $h_{j}(x)$ is either $0$ or
has degree $< k$
and the terms in $h_{j}(x) x^{kj}$ correspond to the 
terms in the expansion of $x^{2\lfloor \varepsilon k \rfloor} W(x) \widetilde{W}(x)$ having degrees in the interval
$[kj,k(j+1))$.  
We see now that \eqref{ffkpap9} is a consequence of \eqref{ffkpap8}.

From \eqref{ffkpap8} and \eqref{ffkpap10}, given any two of $G_{1}(x,y)$, $G_{2}(x,y)$, $H_{1}(x,y)$ 
and $H_{2}(x,y)$, each one will have a factor different from $x$ that
does not divide the other.  The theorem now follows from \eqref{ffkpap9}.  
\end{proof}

%%%%%%%%%%%%%%%%%%%%

\section{Reduction to polynomials in one variable}\label{multitoonesect}
In this section, we will be interested in obtaining information about the factorization of $F(x,x^{n})$ 
in $\mathbb Z[x]$ given $F(x,y)$ as in \eqref{generalpoly}.  As suggested by the previous section, 
the methods we employ here give us results about the factorization of the non-reciprocal part of $F(x,x^{n})$.  

We restrict ourselves to the case that
\[
\gcd{}_{\mathbb Z[x]}\big(f_{0}(x), f_{1}(x), \ldots, f_{r}(x) \big) = 1,
\quad
f_{0}(0) \ne 0,
\quad \text{and} \quad
f_{r}(x) \ne 0.
\]
The above restrictions are minor; in particular, we can simply factor out an element of $\mathbb Z[x]$ from 
$F(x,y)$ if either of the first two conditions do not hold and proceed from there.
To clarify, we allow for the possibility that $f_{j}(0) = 0$ for some or all $j \in \{ 1, 2, \ldots, r \}$ and
$f_{j}(x) = 0$ for some or all $j \in\{ 1, 2, \ldots, r-1 \}$.

Let 
\begin{equation}\label{fbasiceq}
F_{0}(x) = F(x,x^{n}) = \sum_{j=0}^{r} f_{j}(x) x^{n j},
\end{equation}
where we consider here
\begin{equation}\label{nboundtemp}
n > \max_{0 \le j \le r-1}\{ \deg f_{j} \}.
\end{equation}
For the purposes of using Theorem \ref{mainlemma}, we write
\[
F_{0}(x)=\displaystyle\sum_{j=0}^{m} a_{j} x^{d_j},
\]
where $0 = d_{0} < d_{1} < \cdots < d_{m}$, 
the numbers $n, 2n, \ldots, rn$ are included in the set $\{ d_{0}, d_{1}, \dots, d_{m} \}$, and 
possibly the coefficients $a_{j}$ corresponding to $x^{n}, x^{2n}, \ldots, x^{rn}$ are $0$ but otherwise $a_{j} \ne 0$.
With this understanding, 
there exist non-negative integers $\rho_{0}, \ldots, \rho_{r}$ such that 
\begin{gather*}
f_{0}(x)= \sum_{j=0}^{\rho_{0}} a_jx^{d_j}, \quad 
f_{1}(x)= \sum_{j=\rho_{0}+1}^{\rho_{1}} a_{j} x^{d_{j}-n}, \\[5pt]
\quad \ldots,\quad 
f_{r}(x)=\sum_{j=\rho_{r-1}+1}^{\rho_{r}} a_{j} x^{d_{j}-rn},
\end{gather*}
where 
\[
d_{0} = 0, \ 
d_{\rho_{0}+1} = n, \ 
d_{\rho_{1}+1} = 2n, \ 
\ldots, \ 
d_{\rho_{r-1}+1} = rn.
\]
As noted, some or all of the coefficients
$a_{\rho_{0}+1}, a_{\rho_{1}+1}, \ldots, a_{\rho_{r-1}+1}$ may be $0$.  

Set
\[
k_0=2\max_{0 \le j \le r}\{\deg(f_{j})\} 
\quad \text{and} \quad
\varepsilon = \dfrac{1}{2r+2} \le \dfrac{1}{4}.
\]
Choose $n$ sufficiently large so that the inequality on $\deg F_{0}$ holds in Theorem~\ref{mainlemma}.
Suppose further that the non-reciprocal part of $F_{0}(x)$ is reducible. 
Fix $k$ as in Theorem~\ref{mainlemma}. 
For $j\in\{0,\ldots, r\}$, define $\overline{d_j}$ and $\ell_j$ as in Theorem~\ref{mainlemma}.  
Since $k \ge 2\deg(f_{0})$, we have that $d_j+\lfloor \varepsilon k \rfloor < k$ for 
$j\in\{0,1,\ldots, \rho_{0}\}$.  
Hence, each of $\ell_{0}, \ell_{1},$ 
$\ldots, \ell_{\rho_0}$ is 0.  
We show next that, for each $i \in \{ 1, 2, \ldots, r \}$, the numbers 
$\ell_{\rho_{i-1}+1},\ell_{\rho_{i-1}+2},\ldots,\ell_{\rho_{i}}$ are all equal.  
Assume that $\ell_{s} \ne \ell_{\rho_{i}}$ for some 
$s\in\{\rho_{i-1}+1,\rho_{i-1}+2,\ldots, \rho_{i}-1\}$.  
The ordering on the $d_j$ and the definition of the $\ell_j$ imply that 
$\ell_{s} < \ell_{\rho_{i}}$.  By the way $\overline{d_j}$ is defined, we know that 
each $\overline{d_j}$ is in $[0,2\varepsilon k) \subseteq [0,k/2)$.  Hence, by the definition of 
$\ell_j$, we obtain 
\[
d_{\rho_{i}}-d_{s}=k(\ell_{\rho_{i}}-\ell_{s})
+\big(\overline{d_{\rho_{i}}}-\overline{d_{s}}\big)>k-\frac{k}{2}=\frac{k}{2}\geq \deg(f_{i}).
\]
On the other hand, from the definition of the $\rho_{j}$, 
we see that
\[
d_{\rho_{i}}-d_{s} \le d_{\rho_{i}}-d_{\rho_{i-1}+1} \le \deg(f_{i}).
\]
This apparent contradiction implies then that
$\ell_{\rho_{i-1}+1},\ldots,\ell_{\rho_{i}}$ are all equal.

Since the $d_{j}$ are increasing, the equality of the numbers 
$\ell_{\rho_{i-1}+1},\ldots,\ell_{\rho_{i}}$
implies that 
\[
\overline{d_{\rho_{i-1}+1}} = 
\min \big\{\, \overline{d_{\rho_{i-1}+1}}, \overline{d_{\rho_{i-1}+2}}, \ldots, 
\overline{d_{\rho_{i}}}  \,\big\}
\quad \text{for } 1 \le i \le r.
\]
Similarly,
\[
\overline{d_{0}} = \min \big\{\, \overline{d_{0}}, \overline{d_{1}}, 
\ldots, \overline{d_{\rho_{0}}}  \,\big\}
= \lfloor \varepsilon k \rfloor.
\]
Set
\[
M_{0} = \min_{0 \le j \le m} \big\{ \overline{d_j} \big\} 
= \min \{\, \overline{d_{0}}, \overline{d_{\rho_{0}+1}}, \ldots,
\overline{d_{\rho_{r-1}+1}} \,\}.
\]
In the notation of Theorem~\ref{mainlemma}, $M \ge M_{0}$.  Define
\[
t_{i} = \overline{d_{\rho_{i-1}+1}} - M_{0} \ \ \ 
\text{(for $1 \le i \le r$),} \quad \text{and} \quad
t_{0} = \overline{d_{0}} - M_{0} = \lfloor \varepsilon k \rfloor - M_{0}.
\]
Thus, $t_{0}, t_{1}, \ldots, t_{r}$ are non-negative integers $< 2 \varepsilon k$ with at least one 
equal to $0$.

Let $G(x,y)$ be defined as in Theorem~\ref{mainlemma}.  Thus,
\begin{equation}\label{gxyequat}
\begin{aligned}
G(x,y) &= \sum_{j=0}^{m} a_{j} x^{\overline{d_j}} y^{\ell_j} \\
&= \sum_{j=0}^{\rho_{0}} a_{j} x^{\overline{d_j}}
+ \sum_{j=\rho_{0}+1}^{\rho_{1}} a_{j} x^{\overline{d_j}} y^{\ell_{\rho_{1}}}
+ \cdots
+ \sum_{j=\rho_{r-1}+1}^{\rho_{r}} a_{j} x^{\overline{d_j}} y^{\ell_{\rho_{r}}}.
\end{aligned}
\end{equation}
For $0 \le j \le \rho_{0}$, we have
\[
\overline{d_j} = d_{j} + \lfloor \varepsilon k \rfloor = d_{j} +t_{0}+ M_{0}.
\]
For $1 \le i \le r$ and $\rho_{i-1}+1 \le j \le \rho_{i}$, we have
\begin{equation}\label{djbarequation}
\begin{aligned}
\overline{d_j} &= d_{j} + \lfloor \varepsilon k \rfloor - k \ell_{\rho_{i}}
= d_{j} + \overline{d_{\rho_{i-1}+1}} - d_{\rho_{i-1}+1} \\[5pt]
&= d_{j} - d_{\rho_{i-1}+1} + t_{i}+ M_{0}
= d_{j} - i n + t_{i}+ M_{0}.
\end{aligned}
\end{equation}
Observe that, for $1 \le i \le r$, we obtain from \eqref{djbarequation} and the definition of $t_{i}$ that
\begin{equation}\label{eq1djbardeduction}
i n + \lfloor \varepsilon k\rfloor = k \ell_{\rho_{i}} + t_{i} + M_{0}
\qquad \text{with } t_{i}+M_{0} \in [0,2 \varepsilon k).
\end{equation}
As $\ell_{\rho_{0}} = 0$ and $t_{0}+M_{0} = \lfloor \varepsilon k \rfloor$, we see that \eqref{eq1djbardeduction} also holds for $i = 0$.  
Substituting \eqref{djbarequation} and then \eqref{eq1djbardeduction} into \eqref{gxyequat}, we deduce
\begin{align*}
G(x,x^{k}) 
&= f_{0}(x) x^{t_{0}+ M_{0}}
+ f_{1}(x) x^{t_{1}+ M_{0}} x^{k \ell_{\rho_{1}}}
+ \cdots
+ f_{r}(x) x^{t_{r}+ M_{0}} x^{k \ell_{\rho_{r}}} \\[5pt]
&=x^{\lfloor \varepsilon k\rfloor} F_{0}(x).
\end{align*}

Fix $i \in \{1, \ldots, r \}$.  From \eqref{eq1djbardeduction}, we obtain 
and
\begin{equation}
\label{eq2djbardeduction}
(i-1)n + \lfloor \varepsilon k\rfloor = k \ell_{\rho_{i-1}} + t_{i-1} + M_{0} \qquad \text{with}
\qquad t_{i-1} + M_{0}  \in [0,2 \varepsilon k).
\end{equation}
By multiplying equation \eqref{eq2djbardeduction} by $i$ and \eqref{eq1djbardeduction} by $i-1$ and taking a difference, we deduce 
\begin{equation}
\label{eq3}
\lfloor \varepsilon k\rfloor=k(i \ell_{\rho_{i-1}} - (i-1)\ell_{\rho_{i}}) + d'_{i}
\end{equation}
where
\[
d'_{i} = i(t_{i-1} + M_{0}) - (i-1)(t_{i} + M_{0})
\in \big({-2}(r-1)\varepsilon k, 2r\varepsilon k\big).
\]
As $\varepsilon = 1/(2r+2)$, we deduce that $k(i \ell_{\rho_{i-1}} - (i-1)\ell_{\rho_{i}})$ is a multiple of $k$ in the interval $(-k,k)$,
and therefore $0$.  Thus, $i \ell_{\rho_{i-1}} = (i-1)\ell_{\rho_{i}}$ for each $i \in \{1, \ldots, r \}$.  Considering $i = 1, 2, \ldots, r$ 
in turn, we deduce
\[
\ell_{\rho_{0}} = 0 \quad \text{and} \quad
\ell_{\rho_{i}} = i \ell_{\rho_{1}} \quad \text{for } 1 \le i \le r.
\]
Define $\ell = \ell_{\rho_{1}}$.  
Recall that after the statement of Theorem~\ref{mainlemma}, we showed that some $\ell_{j} \ge 1$.  Hence, $\ell \ge 1$.  

From $i \ell_{\rho_{i-1}} = (i-1)\ell_{\rho_{i}}$, \eqref{eq3} and the definition of $d'_{i}$, we deduce
\begin{equation}\label{mzeroteq}
i t_{i-1} - (i-1)t_{i} + M_{0} = \lfloor \varepsilon k\rfloor
\quad 
\text{for } 1 \le i \le r.
\end{equation}
Taking $i = j+1$ and $i = j$ in this equation, subtracting the first from the second, and dividing by $j$, we obtain
\[
2t_{j} = t_{j+1} + t_{j-1}
\quad
\text{ for } 1 \le j \le r-1.
\]
In other words, $t_{j}$ is the average of $ t_{j-1}$ and $t_{j+1}$ for $1 \le j \le r-1$.  
Hence, the minimum value of $t_{i}$ occurs for $i = 0$ or $i = r$.
Recall that the $t_{i}$ are non-negative with at least one equal to $0$.  
Therefore, either $t_{0} = 0$ or $t_{r} = 0$.  Define
\[
t = 
\begin{cases}
t_{1} &\text{if } t_{0} = 0 \\
t_{r-1} &\text{if } t_{r} = 0.
\end{cases}
\]
Since $t_{j}$ is the average of $ t_{j-1}$ and $t_{j+1}$ for $1 \le j \le r-1$, we deduce that either
$t_{i} = i t$ for $0 \le i \le r$ or $t_{i} = (r-i) t$ for $0 \le i \le r$.  

Combining the above, we obtain that one of
\begin{equation}\label{eqforegs1}
x^{M_{0}} \big( f_{0}(x) 
+ f_{1}(x) x^{t} x^{k \ell}
+ \cdots
+ f_{r-1}(x) x^{(r-1)t} x^{(r-1)k \ell}
+ f_{r}(x) x^{rt} x^{r k \ell} \big)
\end{equation}
and
\begin{equation}\label{eqforegs2}
x^{M_{0}} \big( f_{0}(x) x^{rt}
+ f_{1}(x) x^{(r-1)t} x^{k \ell}
+ \cdots
+ f_{r-1}(x) x^{t} x^{(r-1)k \ell}
+ f_{r}(x) x^{r k \ell} \big)
\end{equation}
is equal to $x^{\lfloor \varepsilon k\rfloor} F_{0}(x)$.  
In the first case above, \eqref{mzeroteq} with $i=1$ implies $M_{0} = \lfloor \varepsilon k\rfloor$.  
In the second case, \eqref{mzeroteq} with $i=r$ implies $M_{0} = \lfloor \varepsilon k\rfloor - rt$.  

We focus now on $G(x,y)$.  We have that $G(x,y)$ is either
\[
x^{M_{0}} \big( f_{0}(x) 
+ f_{1}(x) x^{t} y^{\ell}
+ \cdots
+ f_{r-1}(x) x^{(r-1)t} y^{(r-1)\ell}
+ f_{r}(x) x^{rt} y^{r \ell} \big)
\]
or
\[
x^{M_{0}} \big( f_{0}(x) x^{rt}
+ f_{1}(x) x^{(r-1)t} y^{\ell}
+ \cdots
+ f_{r-1}(x) x^{t} y^{(r-1)\ell}
+ f_{r}(x) y^{r \ell} \big).
\]
We deduce that $M$ in Theorem \ref{mainlemma} is $M_{0}$ in the first case, but the value of $M$ depends on whether $f_{r}(0) = 0$ in the second case. 
If $f_{r}(0) \ne 0$, then $M = M_{0}$ in this case as well.  

In any case, to deduce the irreducibility of the the non-reciprocal part of $F(x,x^{n})$ for $n$ sufficiently large, Theorem~\ref{mainlemma} implies that it suffices to show that both
\begin{equation}\label{polyoneeq}
f_{0}(x) 
+ f_{1}(x) x^{t} y^{\ell}
+ \cdots
+ f_{r-1}(x) x^{(r-1)t} y^{(r-1)\ell}
+ f_{r}(x) x^{rt} y^{r \ell}
\end{equation}
and
\begin{equation}\label{polytwoeq}
f_{0}(x) x^{rt}
+ f_{1}(x) x^{(r-1)t} y^{\ell}
+ \cdots
+ f_{r-1}(x) x^{t} y^{(r-1)\ell}
+ f_{r}(x) y^{r \ell},
\end{equation}
other than possibly a factor of a power of $x$, are irreducible as a polynomial in two variables for every non-negative integer $t$.  

We summarize this as follows.

\begin{theorem}\label{newthmtouse}
Let $F(x,y)$ be as in \eqref{generalpoly} with 
\[
\gcd{}_{\mathbb Z[x]}\big(f_{0}(x), f_{1}(x), \ldots, f_{r}(x) \big) = 1
\quad
f_{0}(0) \ne 0,
\quad \text{and} \quad
f_{r}(x) \ne 0.
\]
Let $n$ be sufficiently large.  Then the non-reciprocal part of $F(x,x^{n})$ is not reducible if, 
for every non-negative integer $t$ and every positive integer $\ell$, the polynomials
\eqref{polyoneeq} and \eqref{polytwoeq}
are a power of $x$ (possibly $x^{0}$) times an irreducible polynomial in $\mathbb Z[x,y]$.  
\end{theorem}

Theorem~\ref{newthmtouse} is close to what we stated for Theorem~\ref{themainthm} in the introduction.  
However, if we are not interested in establishing that $F(x,x^{n})$ is irreducible for all 
$n$ sufficiently large but rather for $n$ sufficiently large of a certain form, then 
Theorem~\ref{themainthm} is valuable and needs a little more justification.  
We begin by showing $n = k \ell \pm t$.  

First,  
consider $\eqref{polyoneeq}$.    
We recall the discussion after \eqref{eqforegs1} and \eqref{eqforegs2}.  
In the case of $\eqref{polyoneeq}$, we have $M_{0} = \lfloor \varepsilon k \rfloor$ in \eqref{eqforegs1}.  
The values of $t$ and $\ell$ were determined by
\[
t = t_{1} = \overline{d_{\rho_{0}+1}} - M_{0} = \overline{d_{\rho_{0}+1}} - \lfloor \varepsilon k \rfloor
\quad \text{and} \quad
\ell = \ell_{\rho_{1}} = \ell_{\rho_{0}+1}.
\]
Further, $d_{\rho_{0}+1} = n$.  From Theorem~\ref{mainlemma}, we deduce
\[
n + \lfloor \varepsilon k \rfloor
= d_{\rho_{0}+1} + \lfloor \varepsilon k \rfloor  
= k \ell_{\rho_{0}+1} + \overline{d_{\rho_{0}+1}}
= k \ell + t + \lfloor \varepsilon k \rfloor.
\]
Thus, $n = k \ell + t$.  In the notation of Theorem~\ref{mainlemma} and \eqref{fbasiceq}, 
$F_{0}(x) = F(x,x^{n})$ and the value of $n$ comes from the expression $f_{1}(x) x^{n}$
that one obtains by replacing $y$ with $x^{k}$ in $\eqref{polyoneeq}$.
In the case of $\eqref{polytwoeq}$, we have
\[
0 = t_{r} = \overline{d_{\rho_{r-1}+1}} - M_{0} = \overline{d_{\rho_{r-1}+1}} - \lfloor \varepsilon k \rfloor + rt
\quad \text{and} \quad
\ell_{\rho_{r}} = r \ell.
\]
Thus,
\begin{align*}
r n + \lfloor \varepsilon k \rfloor 
&= d_{\rho_{r-1}+1} + \lfloor \varepsilon k \rfloor
= k \ell_{\rho_{r-1}+1} + \overline{d_{\rho_{r-1}+1}} \\[5pt]
&= k \ell_{\rho_{r}} + \lfloor \varepsilon k \rfloor - rt
= r ( k \ell - t) + \lfloor \varepsilon k \rfloor.
\end{align*}
Hence, in this case, $n = k \ell - t$, and the value of $n$ can be viewed as coming from the expression $f_{1}(x) x^{n}$
that one obtains by dividing the polynomial in $\eqref{polytwoeq}$ by $x^{rt}$ and replacing $y$ with $x^{k}$.

The condition on $k$ in Theorem~\ref{themainthm} follows from our choice of $k \ge k_{0}$.  
Recall that the $t_{j}$ are non-negative integers $< 2\varepsilon k$ and that $\varepsilon = 1/(2r+2)$.  
As $t_{0}, t_{1}, \ldots, t_{r}$ are the same as $0, t, 2t, \ldots, rt$ in some order, we deduce $t < k/(r(r+1))$.  

This completes the justification of Theorem~\ref{themainthm}.

%%%%%%%%%%%%%%%%%%%%

\section{Applications}\label{sectapplications}

In this section, we give three examples using a couple different approaches to show how one
can apply Theorem~\ref{themainthm}.  The first example is meant just to demonstrate the power of
the approach for non-reciprocal polynomials.  
The last two examples come from the Laurent polynomials discussed
in the introduction.  Recalling the first one listed there was handled in \cite{filgar}, the last two examples here
address the two other Laurent polynomials beginning with the last one.

\vskip 5pt \noindent
\textbf{Example 1.}
Let $n \ge 2$ be an integer, and set
\begin{align*}
G(x) = x^{6n} + (x+1) &x^{5n+1} + 2x^{4n} + (x^4 - x^3 - x^2 - 2x - 2)x^{3n-2} \\
&\quad+ 2x^{2n} + (x+1) x^{n-2} + 1.
\end{align*}
One can check that if $n = 4$,  the polynomial factors as a product of two irreducible
non-reciprocal polynomials.  
For $n$ sufficiently large, we show that $G(x)$ has no reciprocal irreducible factors 
and the non-reciprocal part of $G(x)$ is irreducible.  
Then it will follow that $G(x)$ is irreducible for all $n$ sufficiently large.

Assume $u(x)$ is an irreducible reciprocal factor of $G(x)$, 
and let $\alpha$ be a root of $u(x)$.  
Since $G(0) \ne 0$, we see that $\alpha \ne 0$, and the definition of
$u(x)$ being reciprocal implies that $1/\alpha$ is a root of $u(x)$.
We deduce that $\alpha$ is a root of $G(x)$ and of $\widetilde{G}(x) = x^{\deg G} G(1/x)$. 
Since
\begin{align*}
\widetilde{G}(x) = x^{6n} + (x+1) &x^{5n+1} + 2x^{4n} + (-2 x^4 -2 x^3 - x^2 - x +1)x^{3n-2} \\
&\quad+ 2x^{2n} + (x+1) x^{n-2} + 1,
\end{align*}
we deduce that $\alpha$ is a root of 
\[
G(x) - \widetilde{G}(x) = (3 x^4 + x^3 - x - 3) x^{3n-2}.
\]
Therefore, $u(x)$ is an irreducible factor of 
\[
3 x^4 + x^3 - x - 3
= (x - 1)(x + 1)(3x^2 + x + 3).
\]
One checks that $\pm 1$ are not roots of $G(x)$.  
Furthermore, since $G(x) \in \mathbb Z[x]$ is monic, the irreducible polynomial $3x^2 + x + 3$ cannot
be a factor of $G(x)$.  We obtain a contradiction.  Thus, $G(x)$ has no reciprocal irreducible factors.

Let $n$ be sufficiently large.
Let 
\[
F(x,y) = \sum_{j=0}^{6} f_{j}(x) y^{j},
\]
where
\begin{gather*}
f_{0}(x) = 1, \quad
f_{1}(x) = x+1, \quad
f_{2}(x) = 2x^{4}, \\
f_{3}(x) = x^{4} (x^4 - x^3 - x^2 - 2x - 2), \\
f_{4}(x) = 2x^{8}, \quad
f_{5}(x) = x^{11} (x+1) \quad \text{and} \quad
f_{6}(x) = x^{12}.
\end{gather*}
Then $F(x,x^{n-2}) = G(x)$.  

Our goal now is to apply Theorem~\ref{themainthm} or, for this example, Theorem~\ref{newthmtouse}.
Let
\[
F_{1,t}(x,y) = \sum_{j=0}^{6} f_{j}(x) x^{jt} y^{j}
\quad \text{ and } \quad
F_{2,t}(x,y) = \sum_{j=0}^{6} f_{j}(x) x^{(6-j)t} y^{j}.
\]
To show $F_{1,t}(x,y^{\ell})$ and $F_{2,t}(x,y^{\ell})$ are a power of $x$ times an irreducible polynomial for all non-negative integers $t$ and positive integers $\ell$,
we apply Theorem~\ref{thmgeneralpoly}.  
Given the condition in Theorem~\ref{thmgeneralpoly} that $F(x,y)$ is irreducible in $\mathbb Z[x,y]$, we want to first justify that 
$F_{1,t}(x,y)$ and $F_{2,t}(x,y)$ are a power of $x$ times an irreducible polynomial in $\mathbb Z[x,y]$ for all non-negative integers $t$
(so that we are actually applying Theorem~\ref{thmgeneralpoly} to $F_{1,t}(x,y)$ and $F_{2,t}(x,y)$ divided by an appropriate power of $x$).  
Observe that
\[
F_{1,t}\big(x,y/x^{t}\big) = F_{1,0}(x,y)
\quad \text{ and } \quad
F_{2,t}\big(x,x^{t}y\big) = x^{6t} F_{2,0}(x,y).
\]
As $f_{0}(x) = 1$, we see that there are no irreducible polynomials in $\mathbb Z[x]$ dividing all the coefficients of $F_{1,t}(x,y)$ viewed as a polynomial in $y$.  
Also, the polynomial $x$ is the only possible irreducible polynomial in $\mathbb Z[x]$ dividing all the coefficients of $F_{2,t}(x,y)$ viewed as a polynomial in $y$.  
Thus, for $j \in \{ 1, 2 \}$, each irreducible factor of $F_{j,t}(x,y)$ in $\mathbb Z[x,y]$ other than $x$ has positive degree in $y$.  
Therefore, $F_{j,t}(x,y)$ is a power of $x$ times an irreducible polynomial if and only if $F_{j,0}(x,y)$ is,
which follows from the above in $\mathbb Q(x)[y]$ and then by Gauss's Lemma for $\mathbb Z[x,y]$.
One checks directly that the two polynomials $F_{1,0}(x,y)$ and $F_{2,0}(x,y)$ are irreducible.
Therefore, $F_{1,t}(x,y)$ and $F_{2,t}(x,y)$ are a power of $x$ times an irreducible polynomial in $\mathbb Z[x,y]$ for all non-negative integers $t$

Since $f_{1}(x) = x+1$, we see that \eqref{pthpowerequatforF} with $F$ replaced by
$F_{1,t}$ and with $F$ replaced by $F_{2,t}$ does not hold for each prime $p$.  
To finish applying Theorem~\ref{thmgeneralpoly}, 
it remains to show $F_{1,t}(x,y^{4})$ and $F_{2,t}(x,y^{4})$ are a power of $x$ times an irreducible polynomial for all non-negative integers $t$.

Write $t = 4q + r$ where $q \in \mathbb Z$ and $r \in \{ 0,1,2,3 \}$.  
Observe that
\[
F_{1,t}\big(x,(y/x^{q})^{4}\big) = F_{1,r}(x,y^{4})
\quad \text{ and } \quad
F_{2,t}\big(x,(x^{q}y)^{4}\big) = x^{24q} F_{2,r}(x,y^{4}).
\]
The only possible common irreducible factor in $\mathbb Z[x]$ among the coefficients of $F_{1,t}(x,y)$ and $F_{2,t}(x,y)$ 
as polynomials in $y$ is $x$ and only in the case of $F_{2,t}(x,y)$.  
For $j \in \{ 1,2 \}$, a factorization of $F_{j,t}(x,y^{4})$ in $\mathbb Q(x)[y]$ into polynomials of degree $> 0$ in $y$ 
will correspond to a factorization of $F_{j,r}(x,y^{4})$ in $\mathbb Q(x)[y]$ into polynomials of degree $> 0$ in $y$, and vice versa.
We deduce that $F_{j,t}(x,y^{4})$ is a power of $x$ times an irreducible polynomial in $\mathbb Z[x,y]$
if and only if $F_{j,r}(x,y^{4})$ is a power of $x$ times an irreducible polynomial in $\mathbb Z[x,y]$.
As $r \in \{ 0,1,2,3 \}$ and $j \in \{ 1,2 \}$, one can check directly that the eight polynomials $F_{j,r}(x,y^{4})$ 
are a power of $x$ times an irreducible polynomial in $\mathbb Z[x,y]$.  

Given $G(x)$ has no irreducible reciprocal factors,
Theorem~\ref{themainthm} now implies that $G(x)$ is irreducible for all $n$ sufficiently large.

\vskip 5pt \noindent
\textbf{Example 2.}
In the introduction, we mentioned the Laurent polynomial
\[
x^{4b} - x^{2b} - x^{a} - 2 - x^{-a} - x^{-2b} + x^{-4b}
\]
associated with the trace field under $-a/b$ Dehn fillings of the complement of the $4_1$ knot.  
Here, $a$ and $b$ are coprime positive integers.
For fixed $a$ and sufficiently large $b$ or for fixed $b$ and sufficiently large $a$, 
one can apply the prior material to obtain information on 
the factorization of the non-reciprocal part of the polynomial
$\capJ (x^{4b} - x^{2b} - x^{a} - 2 - x^{-a} - x^{-2b} + x^{-4b})$.  

We illustrate this here by considering the case that $a$ is fixed and $b$ is sufficiently large.
Then we are interested in the factorization of
\begin{align*}
F_{0}(x) &= x^{8b} - x^{6b} - x^{4b+a} - 2x^{4b} - x^{4b-a} - x^{2b} + 1 \\
&= x^{4a} x^{4(2b-a)} - x^{3a} x^{3(2b-a)} - (x^{a} + 1)^{2} x^{a} x^{2(2b-a)} - x^{a} x^{2b-a} + 1.
\end{align*}
We define
\[
F(x,y) = f_{0}(x) + f_{1}(x) y +  f_{2}(x) y^{2} +  f_{3}(x) y^{3} +  f_{4}(x) y^{4},
\]
where
\begin{gather*}
f_{0}(x) = 1, \ 
f_{1}(x) = -x^{a}, \ 
f_{2}(x) = - (x^{a} + 1)^{2} x^{a}, \\
f_{3}(x) =  - x^{3a}, \  \text{ and } \ 
f_{4}(x) = x^{4a}.
\end{gather*}
Thus, with $n = 2b-a$, we have $F(x,x^{n}) = F_{0}(x)$.  
Further, define 
\[
F_{1}(x,y) = f_{0}(x) + f_{1}(x) x^{t} y +  f_{2}(x) x^{2t} y^{2} +  f_{3}(x) x^{3t} y^{3} +  f_{4}(x) x^{4t} y^{4}
\]
and
\[
F_{2}(x,y) = f_{0}(x) x^{4t} + f_{1}(x) x^{3t} y +  f_{2}(x) x^{2t} y^{2} +  f_{3}(x) x^{t} y^{3} +  f_{4}(x) y^{4}.
\]
For Theorem~\ref{themainthm}, we want to determine whether the polynomials
$F_{1}(x,y^{\ell})$ and $F_{2}(x,y^{\ell})$
are a power of $x$ times an irreducible polynomial in $\mathbb Z[x,y]$.

For $j \in \{ 1,2 \}$,
we write $F_{j}(x,y) = x^{k_{j}} G_{j}(x,y)$ where $k_{j} \ge 0$ is an integer
and $G_{j}(x,y) \in \mathbb Z[x,y]$ is not divisible by $x$.   
Since $f_{0}(x) = 1$, we see that $k_{1} = 0$.  
We want to apply Theorem~\ref{thmgeneralpoly} to $G_{j}(x,y)$, so 
we show first that $G_{j}(x,y)$ is irreducible.
In the example at the end of Section~\ref{twovarsection}, we saw that
the polynomial
\[
1-x(x+1)y^{n}-2x^2y^{2n}-x^2(x+1)y^{3n}+x^4y^{4n}
\]
is irreducible for every positive integer $n$.
By a change of variables, we see then that
\begin{align*}
1-y(y+1)&x^{a}-2y^2x^{2a}-y^2(y+1)x^{3a}+y^4x^{4a} \\
&= 1 - x^{a} y - (x^{a} + 1)^{2} x^{a} y^{2} - x^{3a} y^{3} + x^{4a} y^{4} \\
&= f_{0}(x) + f_{1}(x) y +  f_{2}(x) y^{2} +  f_{3}(x) y^{3} +  f_{4}(x) y^{4}
\end{align*}
is irreducible for every positive integer $a$.  
Observe that this is the same as $F_{1}(x,y)$ and $F_{2}(x,y)$ with $t = 0$.
We take advantage of this momentarily.

Assume $G_{1}(x,y) = F_{1}(x,y) = U(x,y) V(x,y)$ where $U(x,y)$ and $V(x,y)$ are nonunits in $\mathbb Z[x,y]$.
Since $f_{0}(x) = 1$, the polynomials $f_{0}(x)$, $f_{2}(x) x^{2t}$ and $f_{4}(x) x^{4t}$ have no common factor in $\mathbb Z[x]$, 
so the degrees of $U(x,y)$ and $V(x,y)$ in the variable $y$ are each at least $1$.  
Since
\begin{align*}
 f_{0}(x) + f_{1}(x) y +  f_{2}(x) y^{2} +  f_{3}(x) y^{3} +  f_{4}(x) y^{4} 
&= G_{1}(x,y/x^{t}) \\
&= U(x,y/x^{t}) \,V(x,y/x^{t}),
\end{align*}
we see that $f_{0}(x) + f_{1}(x) y +  f_{2}(x) y^{2} +  f_{3}(x) y^{3} +  f_{4}(x) y^{4}$ is reducible in $\mathbb Q(x)[y]$ and,
hence, in $\mathbb Z[x,y]$.  However, as noted above, $f_{0}(x) + f_{1}(x) y +  f_{2}(x) y^{2} +  f_{3}(x) y^{3} +  f_{4}(x) y^{4}$
is irreducible independent of the choice of $a \ge 1$.  
Thus, we obtain a contradiction, and $G_{1}(x,y)$ is irreducible.

Assume now $x^{k_{2}} G_{2}(x,y) = F_{2}(x,y) = x^{k_{2}} U(x,y) V(x,y)$ where $U(x,y)$ and $V(x,y)$ are nonunits in $\mathbb Z[x,y]$.
As $x$ is not a factor of $G_{2}(x,y) = U(x,y) V(x,y)$ and $f_{0}(x) = 1$, again we have
the degrees of $U(x,y)$ and $V(x,y)$ in the variable $y$ are each at least $1$.  
In this case,
\begin{align*}
x^{4t} \big( f_{0}(x) + f_{1}(x) y +  f_{2}(x) y^{2} +  f_{3}(x) y^{3} &+  f_{4}(x) y^{4} \big)
= F_{2}(x,x^{t}y) \\
&\quad = x^{k_{2}} U(x,x^{t}y) \,V(x,x^{t}y),
\end{align*}
so that again we deduce $f_{0}(x) + f_{1}(x) y +  f_{2}(x) y^{2} +  f_{3}(x) y^{3} +  f_{4}(x) y^{4}$ is reducible in $\mathbb Q(x)[y]$, 
leading to a contradiction.  Thus, $G_{2}(x,y)$ is irreducible.

Now, consider $G_{1}(x,y) = F_{1}(x,y)$.  
Recalling Theorem~\ref{themainthm},
in the case of $F_{1}(x,y^{\ell})$, we have $n = k \ell + t$.  Also, 
$F_{0}(x) = F(x,x^{n}) = F_{1}(x,x^{k\ell})$ so that $n = 2b-a$.  

Assume that $F_{1}(x,y^{4})$ is irreducible but 
$F_{1}(x,y^{\ell})$ is reducible for some positive integer $\ell$.  
By Theorem~\ref{thmgeneralpoly}, there must then be an 
odd prime $p$ dividing $\ell$ for which $F_{1}(x,y^{p})$ is reducible.  Further, as stated
in Theorem~\ref{thmgeneralpoly}, we obtain \eqref{pthpowerequatforF} with $F$ there replaced by $F_{1}$.  
Observe that if $p|a$ and $p|t$, then in fact \eqref{pthpowerequatforF} holds with $F$ replaced by $F_{1}$.  
However, then we have $p|a$, $p|t$ and $p|\ell$ so that $p|(k\ell+t)$ which is equivalent to $p|n$.  
As $n = 2b-a$ and $p$ is an odd prime dividing $a$, we would have $p|b$, contrary to the condition
given that $a$ and $b$ are coprime positive integers.  Thus, either $p \nmid a$ or $p \nmid t$.  
Using \eqref{pthpowerequatforF} with $F$ replaced by $F_{1}$, we see that $f_{1}(x) x^{t}$ is a $p$th power
modulo $p$ and, consequently, that $p | (a+t)$.  Also, we have that $f_{2}(x) x^{2t}$ is a $p$th power
modulo $p$, so its degree as a polynomial modulo $p$, which is $3a + 2t$, must be divisible by $p$.  
On the other hand, $p$ dividing both $a+t$ and $3a + 2t$ implies $p|a$ and $p|t$, giving a contradiction.

For the case of $F_{1}(x,y)$, we are left with determining whether $F_{1}(x,y^{4})$ is reducible. 
Observe that if $F_{1}(x,y^{4})$ is reducible, the remarks after Theorem~\ref{thmgeneralpoly}
immediately give us that $F_{1}(x,y^{2})$ is reducible and \eqref{pthpowerequatforF} holds with 
$F$ there replaced by $F_{1}$ and $p = 2$.  From \eqref{pthpowerequatforF}, we deduce that
$f_{1}(x) x^{t}$ and $f_{2}(x) x^{2t}$ are squares modulo $2$ so that $a+t$ and $a+2t$ are even.  
Hence, we have that both $a$ and $t$ are even.  We make particular use of the latter.

We write $t$ in the form $2q$, where $q$ is a non-negative integer.    
Since $f_{0}(x) = 1$, we see that if $F_{1}(x,y^{2})$ is reducible, then so is $F_{1}(x,(y/x^{q})^{2})$.  
As $F_{1}(x,(y/x^{q})^{2})$ takes the form of $F_{1}(x,y^{2})$ with $t = 0$, we can
restrict our attention to determining whether $F_{1}(x,y^{2})$ is reducible with $t = 0$.  

Set $t = 0$.  
Let $W(x,y) = F_{1}(x,y^{2})$.
We provide here an alternative approach from prior material in this paper for establishing that $W(x,y)$ is irreducible.  
The degree of $F_{1}(x,y^{2})$ in $y$ is $8$.  We consider explicit values of $y \in \mathbb Z$
for which $W(x,y)$ can be checked to be irreducible in $\mathbb Z[x]$.  For example, consider
\[
W(x,2) = 2^{8} x^{4a} - 2^{6} x^{3a} - 2^{4} (x^{a}+1)^{2} x^{a} - 2^{2} x^{a} +1.
\]
We write $W(x,2) = w(x^{a})$ where in this case
\[
w(x) = 256 x^{4} - 80 x^{3} - 32 x^{2} - 20 x + 1.
\]
One checks that the quartic $w(x)$ is irreducible in $\mathbb Z[x]$ (using Maple, for example).  
By Capelli's Theorem, $w(x^{a})$ can be reducible only if $p|a$ for some prime $p$ and $w(x^{p})$ is reducible
or $4|a$ and $w(x^{4})$ is reducible.  A computation shows that $w(x^{4})$ is irreducible, so this second case
does not happen.  In the first case, Capelli's theorem implies we furthermore can assume that if $\gamma$ is a root of $w(x)$, 
then $\gamma$ is a $p$th power in $\mathbb Q(\gamma)$ so that the norm of $\gamma$ is a $p$th power in $\mathbb Q$.  
The norm of a root of $w(x)$ is $1/256 = (1/2)^{8}$.  Hence, in this case $p = 2$.  As we have already verified
that $w(x^{4})$ is irreducible, we know $w(x^{2})$ is irreducible.   Thus, $W(x,2)$ is irreducible for all integers $a \ge 1$.

The idea now is to repeat the above process to obtain the irreducibility of $W(x,k)$ for $17$ different $k \in \mathbb Z$
and for all positive integers $a$.  We used $k \in \{ 2, 3, \ldots, 18  \}$. 
Let $w_{k}(x)$ be the analog to $w(x)$ above so that $w_{k}(x^{a}) = W(x,k)$.  
For each $k$, we verified computationally the irreducibility of $w_{k}(x^{a})$ for $a \in \{ 1, 3, 4 \}$.  
As the norm of a root of $w_{k}(x)$ is $1/k^{8}$ and this can be a $p$th power for $k \in \{ 2, 3, \ldots, 18  \}$ only for
$p \in \{ 2, 3 \}$, we deduce as above that $W(x,k)$ is irreducible for each $k \in \{ 2, 3, \ldots, 18  \}$ and 
for all positive integers $a$.  

Assume that there is a positive integer $a$ for which $W(x,y)$ is reducible in $\mathbb Z[x,y]$.  
As the constant term of $W(x,y)$ is $1$, we deduce that 
$W(x,y) = F_{1}(x,y^{2}) = U(x,y) V(x,y)$ where $U(x,y)$ and $V(x,y)$ are in $\mathbb Z[x,y]$
and each of degree $\ge 1$ in $y$.  Denote these degrees by $d_{U}$ and $d_{V}$, respectively.   
For each $k \in \{ 2, 3, \ldots, 18 \}$, the
irreducibility of $W(x,k) = U(x,k) V(x,k)$ in $\mathbb Z[x]$ implies that one of $U(x,k)$ or $V(x,k)$ is $\pm 1$.  
Since $W(x,y) = U(x,y) V(x,y)$, we have $d_{U} + d_{V} = 8$.  The pigeon-hole principle gives
us that there is a subset $S$ of $\{ 2, 3, \ldots, 18 \}$ such that one of the following holds:
\begin{itemize}\setlength{\itemsep}{0pt}
\item
The size of $S$ is $d_{U}+1$ and $U(x,k) = 1$ for all $k \in S$.
\item
The size of $S$ is $d_{U}+1$ and $U(x,k) = -1$ for all $k \in S$.
\item
The size of $S$ is $d_{V}+1$ and $V(x,k) = 1$ for all $k \in S$.
\item
The size of $S$ is $d_{V}+1$ and $V(x,k) = -1$ for all $k \in S$.
\end{itemize}
We consider the first possibility above, noting the argument in each case is similar.  Writing
$S = \{ k_{0}, k_{1} \ldots, k_{s} \}$ and 
$U(x,y) = u_{0}(x) + u_{1}(x)y + \cdots + u_{s}(x) y^{s}$ where $s = d_{U}$, we obtain
\[
u_{0}(x) + u_{1}(x)k_{j} + \cdots + u_{s}(x) k_{j}^{s} = 1 \quad
\text{for } 0 \le j \le s.
\]
Thus, 
\[
\begin{pmatrix}
1 & k_{0} & \cdots & k_{0}^{s} \\
1 & k_{1} & \cdots & k_{1}^{s} \\
\vdots & \vdots & \ddots & \vdots \\
1 & k_{s} & \cdots & k_{s}^{s} 
\end{pmatrix}
\begin{pmatrix}
u_{0}(x) \\
u_{1}(x) \\
\vdots \\
u_{s}(x)
\end{pmatrix}
= \begin{pmatrix}
1 \\
1 \\
\vdots \\
1
\end{pmatrix}.
\]
The matrix above is a Vandermonde matrix, and its determinant is non-zero.  Hence, there
are unique $u_{0}(x), u_{1}(x), \ldots, u_{s}(x)$ satisfying the matrix equation above.  
Since the values $u_{0}(x) = 1$ and $u_{j}(x) = 0$ for $1 \le j \le s$ satisfy the equation,
we deduce that $U(x,y) = 1$, contradicting that $d_{U} \ge 1$.   
Therefore, $W(x,y)$ is irreducible in $\mathbb Z[x,y]$ for all positive integers $a$.

So far we have shown that $F_{1}(x,y^{\ell})$ is irreducible in $\mathbb Z[x,y]$ for all 
integers $t \ge 0$ and $\ell \ge 0$ for which $n = 2b-a = k\ell + t$.  We now turn to 
$F_{2}(x,y^{\ell})$ where the integers $t \ge 0$ and $\ell \ge 0$ that we are interested 
in satisfy $n = 2b-a = k\ell - t$.   We give a similar argument to that just given for $F_{1}(x,y)$.  
In the case of $F_{2}(x,y)$, we have $F_{2}(x,y) = x^{k_{2}} G_{2}(x,y)$ and $k_{2}$ may be positive
(and will be unless $t = 0$).  We are interested 
in establishing that $F_{2}(x,y^{\ell})$ is a power of $x$ times an irreducible polynomial.
In other words, we want to show $G_{2}(x,y^{\ell})$ is irreducible.
We note that the explicit values of $f_{j}(x)$ give that 
\[
k_{2} = 
\begin{cases}
4t &\text{if } 0 \le t \le a/2 \\
a+2t &\text{if } a/2 < t \le 3a/2 \\
4a &\text{if } t > 3a/2.
\end{cases}
\]
Fix an odd prime $p$ satisfying \eqref{primebound}.
To simplify notation, we set $u = k_{2}$.  
Observe that $u$ is independent of $p$ and $u \le 4t$.  
Assume that $F_{2}(x,y^{p})/x^{u} = G_{2}(x,y^{p})$ is a reducible polynomial.  By Theorem~\ref{thmgeneralpoly},
we obtain that \eqref{pthpowerequatforF} holds.  
Looking at the constant term and the coefficients of $y$ and $y^{2}$ in $F_{2}(x,y^{p})/x^{u}$, 
we see that $p|(4t-u)$, $p|(a+3t-u)$ and $p|(3a+2t-u)$.  These imply $p|a$, $p|t$ and $p|u$.  
As in the case of $F_{1}(x,y^{p})$, this leads to $p$ dividing both $a$ and $b$, contradicting that
$a$ and $b$ are relatively prime.

Defining $u$ as above, we assume now that $F_{2}(x,y^{4})/x^{u} = G_{2}(x,y^{4})$ is reducible.  
As with $F_{1}(x,y)$, we can deduce from the comments after Theorem~\ref{thmgeneralpoly}
that $F_{2}(x,y^{2})/x^{u}$ is reducible.  Taking $p = 2$ in our discussion above about
$F_{2}(x,y^{p})/x^{u}$, we deduce $2|t$, $2|u$ and $2|a$ (but not necessarily that $2|b$).  
Writing $t = 2q$, we see that $F_{2}(x,(x^{q}y)^{2})/x^{8q}$ is of the form  
$F_{2}(x,y^{2})$ with $t = 0$.  Furthermore, as the coefficients of $F_{2}(x,y^{2})/x^{u}$
as a polynomial in $y$ have no common irreducible factor in $\mathbb Z[x]$, we can conclude
that if $F_{2}(x,y^{2})/x^{u}$ is reducible, then $F_{2}(x,y^{2})$ with $t = 0$ is reducible.  

With $t = 0$, we have
\[
F_{2}(x,y) = F_{1}(x,y) = F(x,y).
\]
As we have already established in this case that $F_{1}(x,y^{2})$ is irreducible for all positive integers $a$, 
we deduce that $F_{2}(x,y^{2})$ is irreducible for all positive integers $a$.

We therefore have that any polynomial $F_{1}(x,y)$ or $F_{2}(x,y)$ arising from $F_{0}(x)$ is a power of $x$ times
an irreducible polynomial in $\mathbb Z[x,y]$.  
Theorem~\ref{themainthm} now implies that the non-reciprocal part of $F_{0}(x)$ is not reducible for each fixed positive integer $a$ and 
every sufficiently large positive integer $b$ relatively prime to $a$.  The phrase ``not reducible'' is
deliberate here, as $F_{0}(x)$ is a reciprocal polynomial and cannot have exactly one irreducible non-reciprocal 
factor.  The fact that the non-reciprocal part of $F_{0}(x)$ is not reducible implies then that the non-reciprocal
part of $F_{0}(x)$ is identically $1$.  In other words, we deduce that $F_{0}(x)$ is a product of irreducible 
reciprocal polynomials (for fixed $a$ and sufficiently large $b$ relatively prime to $a$).

\vskip 5pt \noindent
\textbf{Example 3.}
We consider
\[
F(x) = x^{4n}+\left(-x^3+4x^2-8+4x^{-2}-x^{-3}\right)x^{2n}+1,
\]
associated with the trace field of the Whitehead link for the complement of the $(-2,3,3+2n)$ pretzel knot $K_n$.  
This polynomial $F(x)$ is in $\mathbb Z[x]$ for all $n \ge 2$, and 
one easily checks that $(x-1)^{2}$ is a factor.  
A computation suggests that the other factor is always irreducible.  
Observe that the polynomial $x^{2n}+\left(-x^3+4x^2-8+4x^{-2}-x^{-3}\right)x^{n}+1$ with more general exponents
does not have this property.   For example, for $n = 7$, this polynomial factors as
\[
(x^4+x^3+3x^2+x+1)(x^6-x^5+2x^3-x+1)(x-1)^2 (x+1)^2.
\]
It is reasonable to obtain a similar result below for this more general polynomial; however, some adjustments would
be needed as the analog to \eqref{egtwoeq} below is reducible for $r = 0$ so that one should look at the factors that 
arise instead and apply Theorem~\ref{thmgeneralpoly} to them. 

We show here that, for $n$ sufficiently large, $F(x)$ is a product of reciprocal irreducible polynomials.  
We rewrite
\[
F(x) = f_{0}(x) + f_{1}(x) x^{n-2}  + f_{2}(x) x^{2(n-2)} + f_{3}(x) x^{3(n-2)}  + f_{4}(x) x^{4(n-2)},
\]
where
\begin{gather*}
f_{0}(x) = 1, \ 
f_{1}(x) = 0, \ 
f_{2}(x) = -x^7+4x^6-8x^{4}+4x^{2}-x, \\
f_{3}(x) = 0, \  \text{ and } \ 
f_{4}(x) = x^{8}.
\end{gather*}
To apply Theorem~\ref{themainthm}, we want to determine whether the polynomials
$F_{1}(x,y^{\ell})$ and $F_{2}(x,y^{\ell})$ are a power of $x$ times an irreducible polynomial in $\mathbb Z[x,y]$,
where
\[
F_{1}(x,y) = f_{0}(x) 
+ f_{2}(x) x^{2t} y^{2}
+ f_{4}(x) x^{4t} y^{4}
\]
and
\[
F_{2}(x,y) = f_{0}(x) x^{4t}
+ f_{2}(x) x^{2t} y^{2}
+ f_{4}(x) y^{4}.
\] 
We make use of Theorem~\ref{thmgeneralpoly}.

First, we justify that $F_{1}(x,y)$ and $F_{2}(x,y)$ are a power of $x$ times an irreducible polynomial in $\mathbb Z[x,y]$
for all non-negative integers $t$.  In other words, we justify that
\[
F_{j}(x,y) = x^{k_{j}} G_{j}(x,y), 
\]
where $k_{j} \ge 0$ is an integer and
$G_{j}(x,y) \in \mathbb Z[x,y]$ is irreducible for $j \in \{ 1, 2 \}$.  
Assume that $F_{1}(x,y) = U(x,y) V(x,y)$, where $U(x,y)$ and $V(x,y)$ are nonunits in $\mathbb Z[x,y]$.  
Since $f_{0}(x) = 1$, we see that each of $U(x,y)$ and $V(x,y)$ has degree at least $1$ in $y$.
Thus, $F_{1}(x,y/x^{t}) = U(x,y/x^{t}) V(x,y/x^{t})$ provides a non-trivial factorization of 
$f_{0}(x) + f_{2}(x) y^{2} + f_{4}(x) y^{4}$ in $Q(x)[y]$.  However, one can check directly that 
\[
f_{0}(x) + f_{2}(x) y^{2} + f_{4}(x) y^{4}
= 1 + (-x^7+4x^6-8x^4+4x^{2}-x)y^{2} + x^8y^4
\]
is irreducible in $Q(x)[y]$, giving a contradiction.  So $F_{1}(x,y)$ is irreducible in $\mathbb Z[x,y]$;
in other words, one can take $k_{1} = 0$ and $G_{1}(x,y) = F_{1}(x,y)$.  
Observe that $F_{2}(x,y)$ is divisible by $x^{8}$ for $t \ge 4$, so $k_{2} \ne 0$ in general.  
Since $k_{2}$ is in fact the exponent on the largest power of $x$ dividing every coefficient of $F_{2}(x,y)$ viewed as a polynomial in $y$, 
we see that $k_{2} = 2t$ if $t < 4$ and $k_{2} = 8$ if $t \ge 4$.
With this value of $k_2$, assume now that $F_{2}(x,y) = x^{k_{2}} U(x,y) V(x,y)$, where $U(x,y)$ and $V(x,y)$ are nonunits in $\mathbb Z[x,y]$.  
Given the definition of $k_{2}$ and that $f_{0}(x) = 1$, we see that each of $U(x,y)$ and $V(x,y)$ has degree at least $1$ in $y$.  
In this case, we see that 
\[
x^{4t} \big( f_{0}(x) + f_{2}(x) y^{2} + f_{4}(x) y^{4} \big)
= F_{2}(x,x^{t}y) = x^{k_{2}} U(x,x^{t}y) V(x,x^{t}y)
\] 
leads to a non-trivial factorization of 
$f_{0}(x) + f_{2}(x) y^{2} + f_{4}(x) y^{4}$ in $Q(x)[y]$, which again is a contradiction.
Thus, $F_{2}(x,y) = x^{k_{2}} G_{2}(x,y)$, where $G_{2}(x,y)$ is irreducible in $\mathbb Z[x,y]$.  
We are now ready to apply Theorem~\ref{thmgeneralpoly} to $G_{j}(x,y)$.
 
Let $g(x) = -x^6+4x^5-8x^{3}+4x-1$, so $f_{2}(x) = x g(x)$.  For each odd prime $p$, observe that
$g(0) \equiv -1 \pmod{p}$ and the degree of $g(x)$ is $6$ modulo $p$.   
Hence, $f_{2}(x)$ times a power of $x$ cannot be a $p$th power modulo $p$ if $p > 7$.  
A direct check shows that $g(x)$ is furthermore
not a $p$th power modulo $p$ for $p \in \{ 3, 5 \}$, and hence the same holds for $f_{2}(x)$ times a power of $x$.  
It follows that $G_{1}(x,y)$ and $G_{2}(x,y)$ do not satisfy \eqref{pthpowerequatforF} for any odd prime $p$.

To finish applying Theorem~\ref{thmgeneralpoly}, we are left with considering the factorization of
$G_{1}(x,y^{4})$ and $G_{2}(x,y^{4})$.
We write $t = 4 q + r$, where $r \in \{ 0, 1, 2, 3 \}$.  Then
\[
G_{1}(x,y^{4}) =
F_{1}(x,y^{4}) = f_{0}(x) 
+ f_{2}(x) x^{8q+2r} y^{8}
+ f_{4}(x) x^{16q+4r} y^{16}.
\]
Assume $G_{1}(x,y^{4}) = U(x,y) V(x,y)$ with $U(x,y)$ and $V(x,y)$ in $\mathbb Z[x,y]$ with neither $\pm 1$.  
As $\gcd{}_{\mathbb Z[x]}\big(f_{0}(x), f_{2}(x), f_{4}(x) \big) = 1$, each of $U(x,y)$ and $V(x,y)$ has degree
at least $1$ in $y$.  Observe that
\begin{equation}\label{egtwoeq}
U(x, y/x^{q}) V(x, y/x^{q}) = f_{0}(x) 
+ f_{2}(x) x^{2r} y^{8}
+ f_{4}(x) x^{4r} y^{16}.
\end{equation}
We deduce that $f_{0}(x) + f_{2}(x) x^{2r} y^{8} + f_{4}(x) x^{4r} y^{16}$ 
is reducible in $\mathbb Q(x)[y]$ and, hence, in $\mathbb Z[x,y]$.  As $r \in \{ 0, 1, 2, 3 \}$, 
this is easily checked by a computation.  As the $4$ polynomials $f_{0}(x) + f_{2}(x) x^{2r} y^{8} + f_{4}(x) x^{4r} y^{16}$
for $r \in \{ 0, 1, 2, 3 \}$ are irreducible in $\mathbb Z[x,y]$, we conclude that $F_{1}(x,y^{4})$ is irreducible.

The substitution $t = 4 q + r$, where $r \in \{ 0, 1, 2, 3 \}$ in $F_{2}(x,y^{4})$ gives
\[
F_{2}(x,y^{4}) = f_{0}(x) x^{16q+4r}
+ f_{2}(x) x^{8q+2r} y^{8}
+ f_{4}(x) y^{16}.
\]
Note that $F_{2}(x,y^{4})$ is divisible by $f_{4}(x) = x^{8}$ if $q \ge 1$.  We are wanting to know if 
$G_{2}(x,y^{4})$ is an irreducible polynomial in $\mathbb Z[x,y]$.  Assume that
$F_{2}(x,y^{4}) = x^{k_{2}} G_{2}(x,y^{4}) = x^{k_{2}} U(x,y) V(x,y)$ with $U(x,y)$ and $V(x,y)$ in $\mathbb Z[x,y]$ with neither $\pm 1$ and neither divisible by $x$.  
As before, each of $U(x,y)$ and $V(x,y)$ has degree at least $1$ in $y$.  Also, 
\[
x^{k_{2}} U(x, x^{q} y) V(x, x^{q} y) =  x^{16q} \big( f_{0}(x) x^{4r}
+ f_{2}(x) x^{2r} y^{8}
+ f_{4}(x) y^{16} \big),
\]
and we deduce that $f_{0}(x) x^{4r} + f_{2}(x) x^{2r} y^{8} + f_{4}(x) y^{16}$ factors as a power of $x$ times a product of two polynomials in $\mathbb Z[x,y]$, each
with degree $\ge 1$ in $y$.  
We check that this is not the case so that $G_{2}(x,y^{4})$ is irreducible.

Hence, Theorem~\ref{thmgeneralpoly} implies 
$F_{1}(x,y^{\ell})$ and $F_{2}(x,y^{\ell})$ are a power of $x$ times an irreducible polynomial in $\mathbb Z[x,y]$
for every positive integer $\ell$.  
Then Theorem~\ref{themainthm} implies that the non-reciprocal part of $F(x)$ is not reducible for $n$ sufficiently large.
Since $F(x)$ is a reciprocal polynomial, we conclude that $F(x)$ is a product of irreducible reciprocal polynomials.

\subsection*{Acknowledgements}
The author thanks Stavros Garoufalidis for discussions which helped motivate this study.  
He is also grateful for the encouragement and interest Andrzej Schinzel expressed to him on this research
during a conversation in Pozna\'n, Poland, in 2017.

%%%%%%%%%%% To ease editing, use normal size for the references:

\end{document}